\documentclass[12pt]{article}
\usepackage{amsmath}
\usepackage{amssymb}
\usepackage{amsthm}
\usepackage{latexsym}
\usepackage{cite}
\usepackage{psfrag}
\usepackage{epsfig}
\usepackage{graphicx}
\usepackage[latin1]{inputenc}
\usepackage{url}
\usepackage{mathptmx}
\usepackage{amsfonts}
\usepackage{amscd}
\usepackage{verbatim}
\usepackage[misc]{ifsym}
\usepackage{mathrsfs}

\makeatletter
\@addtoreset{equation}{section}
\makeatother

\newcommand{\PG}{\mathrm{PG}}

\newcommand{\F}{\mathbb{F}}
\newcommand{\Hh}{\mathbf{H}}
\newcommand{\B}{\mathbf{B}}
\newcommand{\W}{\mathbf{W}}
\newcommand{\bb}{\mathbf{b}}
\newcommand{\cc}{\mathbf{c}}
\newcommand{\h}{\mathbf{h}}
\newcommand{\f}{\mathbf{f}}
\newcommand{\s}{\mathbf{s}}
\newcommand{\tT}{\mathbf{t}}
\newcommand{\w}{\mathbf{w}}
\newcommand{\x}{\widehat{x}}
\newcommand{\M}{\mathcal{M}}

\newtheorem{theorem}{Theorem}
\newtheorem{lemma}{Lemma}

\newtheorem{proposition}{Proposition}
\newtheorem{conjecture}{Conjecture}

\newtheorem{problem}{Problem}
\theoremstyle{definition}

\newtheorem{definition}{Definition}
\newtheorem{remark}{Remark}

\begin{document}
\title{New covering codes of radius $R$, codimension $tR$ and $tR+\frac{R}{2}$, and saturating sets in projective spaces}
\date{}
\author{ Alexander A. Davydov\footnote{The research of A.A.~Davydov was done at IITP RAS and supported by the Russian Government (Contract No 14.W03.31.0019).} \\
{\footnotesize Institute for Information Transmission Problems
(Kharkevich
institute), Russian Academy of Sciences}\\
{\footnotesize Bol'shoi Karetnyi per. 19, Moscow,
127051, Russian Federation. E-mail: adav@iitp.ru}
\and Stefano Marcugini\footnote{The research of  S. Marcugini and F.~Pambianco was supported in part
by the Italian National Group for Algebraic and Geometric Structures and their Applications (GNSAGA - INDAM)
and by University of Perugia, (Project: "Strutture Geometriche, Combinatoria e loro Applicazioni", Base Research Fund 2017).} and Fernanda Pambianco$^\dag$  \\
{\footnotesize Dipartimento di Matematica e Informatica,
Universit\`{a}
degli Studi di Perugia, }\\
{\footnotesize Via Vanvitelli~1, Perugia, 06123, Italy. E-mail:
\{stefano.marcugini,fernanda.pambianco\}@unipg.it}}
\maketitle
\textbf{Abstract.}
The length function $\ell_q(r,R)$ is the smallest
length of a $ q $-ary linear code of codimension $r$ and covering radius $R$. In this work we obtain new constructive upper bounds on
$\ell_q(r,R)$ for all $R\ge4$, $r=tR$, $t\ge2$, and also for all even $R\ge2$, $r=tR+\frac{R}{2}$, $t\ge1$. The new bounds are provided by infinite families of new covering codes with fixed $R$ and increasing codimension. The new bounds  improve upon the known ones.

We propose a general regular construction (called ``Line+Ovals'') of a minimal $\rho$-saturating $((\rho+1)q+1)$-set in the projective space $\PG(2\rho+1,q)$ for all $\rho\ge0$. Such a set corresponds to an $[Rq+1,Rq+1-2R,3]_qR$ locally optimal\footnote{See the definitions in Sect. \ref{sec_intr}.} code of covering radius $R=\rho+1$.
Basing on combinatorial properties of these codes regarding to spherical capsules, we give constructions for code codimension lifting and obtain infinite families of new surface-covering$^1$ codes with codimension $r=tR$, $t\ge2$.

In addition, we obtain new 1-saturating sets in the projective plane $\PG(2,q^2)$ and, basing on them, construct infinite code families  with fixed even radius $R\ge2$ and codimension $r=tR+\frac{R}{2}$, $t\ge1$.



\section{Introduction}\label{sec_intr}
Let $\F_{q}$ be the Galois field with $q$ elements, $\F_{q}^*=\F_{q}\setminus\{0\}$. Let $\F_{q}^{\,n}$ be the $n$-dimensional vector space over $\F_{q}.$ Denote by $[n,n-r]_{q}$ a $q$-ary
linear code of length $n$ and codimension (redundancy)~$r$, that is a
subspace of $\F_{q}^{\,n}$ of dimension $n-r.$  

Let $d(v,c)$ be the Hamming distance between vectors $v$ and $c$ of $\F_{q}^{\,n}$. The \emph{sphere of radius $R$} with center $c$ in $\F_{q}^{\,n}$
is the set $\{v:v\in \F_{q}^{\,n},$ $d(v,c)\leq R\}$. For $0\leq \ell \leq R$, a \emph{spherical $(R,\ell)$-capsule} with center $c$ in $\F_{q}^{\,n}$
is the set $\{v:v\in \F_{q}^{\,n},$ $\ell\le d(v,c)\leq R\}$ \cite[Rem.\,5]{DavPPI}, \cite[Rem.\,2.1]{Dav95}, \cite[Sect.\,2]{DGMP-AMC}. An \emph{$(R,R)$-capsule is the surface of a sphere of radius $R$.}

\begin{definition} \label{Def1_CoverRad}
 A linear $[n,n-r]_{q}$ code has \emph{covering radius} $R$ and is denoted as an $[n,n-r]_{q}R$ code if any of the following equivalent properties holds:

  \textbf{(i)} The value $R$
 is the least integer such that the space $\F_{q}^{\,n}$ is covered by
the spheres of radius $R$ centered at the codewords.

\textbf{(ii)}
Every column of $\F_{q}^{\,r}$ is equal to a linear combination of at most $R$ columns
of a parity check matrix of the code, and $R$ is the smallest value with
this property.
\end{definition}

An $[n,n-r]_{q}R$ code of minimum distance $d$ is denoted by $[n,n-r,d]_{q}R$ code. For an introduction to coverings of Hamming
spaces, see \cite{Handbook-coverings,CHLS-bookCovCod}.
For fixed $q,r$, and $R$, the covering quality of an $[n, n-r]_{q}R$ code is better if its length~$n$ is smaller.

\begin{definition} \cite{Handbook-coverings,CHLS-bookCovCod}
\emph{The length function} $\ell_q(r,R)$ is the smallest
length of a $ q $-ary linear code of codimension~$r$ and covering radius $R$.
\end{definition}

 It can be shown, see e.g. \cite{BDGMP-R2R3CC_2019,DGMP-AMC}, that
 if  code length $n$ is considerably larger than $R$ (this is the natural case in covering codes investigations) and if $q$ is large enough, then there is a lower bound of the form $\ell_q(r,R)\gtrsim cq^{(r-R)/R},$
 where $c$ is independent of $q$ but it is possible that $c$  depends  on $r$ and $R$.

Let $t,s,R^*$ be integers. Let $q'$ be a prime power.   Consider the following cases:
  \begin{align}\label{eq1_situat}
  \textbf{(i)}~ r=tR,~  \text{arbitrary }q.~~ \textbf{(ii)}~
    R=sR^*,~ r=tR+s,~ q=(q')^{R^*}.~~\textbf{(iii)}~r\ne tR,~ q=(q')^R.
  \end{align}
In \cite{DavCovRad2,DGMP_ACCT2008,DGMP-AMC,DavOst-IEEE2001}, for all the cases in \eqref{eq1_situat},  codes with  lengths  close (by order) to the bound  $\ell_q(r,R)\gtrsim cq^{(r-R)/R}$  are obtained. These lengths are upper bounds on $\ell_q(r,R)$.

\emph{The  goal of this paper} is to improve  on the known upper bounds on $\ell_q(r,R)$ in the case (i) of \eqref{eq1_situat} for $R\ge4$ and in the case (ii) of \eqref{eq1_situat} for even $R$ with $R^*=2$.

The following properties of codes are useful for obtaining new bounds.

\begin{definition}\label{def1_LO} \cite{DFMP-IEEE-LO}
  A linear covering code is called \emph{locally optimal}
if one cannot remove any column\,from its parity check matrix without an increase in covering radius.
\end{definition}

\begin{definition}\label{def1_R,l code} \cite{DavPPI}, \cite[Sect.\,2]{Dav95}, \cite[Sect,\,2]{DGMP-AMC} Let $0\leq \ell \leq R$. An $[n,n-r]_{q}R$ code  is called an \emph{$(R,\ell)$-object} and is denoted by $[n,n-r]_{q}R,\ell$ code if any of the following equivalent conditions holds:

\textbf{(i)} The space $\F_{q}^{\,n}$ is covered by the spherical $(R,\ell)$-capsules centered at the codewords.

\textbf{(ii)} Every column of the space $\F_{q}^{\,r}$ (including
the zero column) is equal to a linear combination with \emph{nonzero
coefficients} of at least $\ell $ and at most $R$ distinct columns of
a parity-check matrix of the code.

\textbf{(iii)} Every coset of the code (including the code itself) contains a weight $w$ word of the space $\F_{q}^{\,n}$ such that $\ell\le w\le R$.
\end{definition}
\begin{definition}\label{def1_SurfCov}
  An $[n,n-r]_{q}R,R$ code is called \emph{surface-covering code} of radius $R$.
\end{definition}

Note that the space $\F_{q}^{\,n}$ is covered by \emph{the surfaces} of the spheres of radius $R$ centered at the codewords of an $[n,n-r]_{q}R,R$ surface-covering code.

  Codes with radius $R=2,3$ and codimension $r=tR$ have been widely investigated, see
  \cite{DavParis,Dav95,DavCovRad2,DFMP-IEEE-LO,DGMP-AMC,DGMP_ACCT2008,DavOst-IEEE2001,DavOstEJC,DavOst-DESI2010} and the references therein. At the same time,
  codes with $R\ge4$, $r=tR$, have not been extensively studied. The main known results for codes with $R\ge4$, $r=tR$, are available in \cite{DGMP_ACCT2008,DGMP-AMC,DavOst-IEEE2001} and collected in Proposition~\ref{prop1_R>=4}.
 \begin{proposition}\label{prop1_R>=4}
  \emph{\cite{DGMP_ACCT2008},\cite[Ths. 6.1,6.2, eqs. 6.1,6.2]{DGMP-AMC},\cite{DavOst-IEEE2001}} The following constructive upper bounds on the length function hold:
\begin{align}\label{eq1_known bound}
&\ell_q(r,R)\le Rq^{(r-R)/R}+\left\lceil \frac{R}{3}\right\rceil q^{(r-2R)/R}+\delta_q(r,R),~R\ge4,~r=tR,~t\ge2,
\end{align}
where $\delta_q(r,R)=0$ if $q\ge4,r=2R$, or $q=16,q\ge23,r=3R$, or $q\ge7,q\neq9,r\ge5R,r\ne6R$. Also, $\delta_q(r,R)=(2R\bmod3)\cdot(q^{(r-3R)/R}+1)$ if $q\ge7,q\neq9,r=4R,6R$.
 \end{proposition}

The main known results for codes with even covering radius $R\ge2$ and codimension $r=tR+\frac{R}{2}$ are available in \cite{DavCovRad2,DGMP_ACCT2008,DGMP-AMC} and collected in Proposition~\ref{prop1_evenR}.
  \begin{proposition}\label{prop1_evenR}
\emph{\cite[Ex.\,6, eq.\,(33)]{DavCovRad2}, \cite{DGMP_ACCT2008}, \cite[Sects.\,4.4,\,7]{DGMP-AMC}}
 Let $q'$ be a prime power. Let the covering radius $R\ge2$ be even. Let the code codimension be $r=tR+\frac{R}{2}$ with integer~$t$. The following constructive upper bounds on the length function hold:
\begin{align}
&\ell_q(r,R)\le\frac{R}{2}\left(3-\frac{1}{\sqrt{q}}\right)q^\frac{r-R}{R}+\frac{R}{2}\left\lfloor q^{(r-2R)/R-0.5}\right\rfloor,~ q=(q')^2\ge16,~t\ge1;\displaybreak[3]\label{eq1_h=1}\\
&\ell_q(r,R)\le R\left(1+\frac{1}{\sqrt[4]{q}}+\frac{1}{\sqrt{q}}\right)q^\frac{r-R}{R}+\frac{R}{2}\left\lfloor q^{(r-2R)/R-0.5}\right\rfloor,~ q=(q')^4,~t\ge1;\displaybreak[3]\label{eq1_h=2}\\
&\ell_q(r,R)\le R\left(1+\frac{1}{\sqrt[6]{q}}+\frac{1}{\sqrt[3]{q}}+\frac{1}{\sqrt{q}}\right)q^{(r-R)/R}+R\left\lfloor q^{(r-2R)/R-0.5}\right\rfloor,~ q=(q')^6,\label{eq1_h=3}\displaybreak[3]\\
&q'\le73 \mbox{ prime},~t\ge1,~t\ne4,6.\notag
\end{align}
  \end{proposition}
\begin{problem}\label{probl1}\looseness=-1
  Improve on the known bounds on the length function $\ell_q(r,R)$ collected in

  \textbf{(i)} Proposition \ref{prop1_R>=4} where $R\ge4$, $r=tR$, $t\ge2$;

   \textbf{(ii)} Proposition \ref{prop1_evenR} where $R\ge2$, $r=tR+\frac{R}{2}$, $t\ge1$.
\end{problem}

Effective methods to obtain upper bounds on $\ell_q(r,R)$ are connected with \emph{saturating sets in projective spaces}.
Let $\PG(N,q)$ be the $N$-dimensional projective space over the field $\F_q$; see \cite{Hirs,HirsSt-old,HirsStor-2001} for an introduction to the projective spaces  and \cite{EtzStorm2016,Giul2013Survey,HirsSt-old,Klein-Stor,LandSt} for connections  between coding theory and Galois geometries.

\begin{definition}
\label{def1_usual satur}

 A point set $S\subseteq\PG(N,q)$ is
$\rho$-\emph{saturating} if any of the following equivalent properties holds:

\textbf{(i)} For any point $A$ of\/ $\PG(N,q)\setminus S$
there exist $\rho+1$ points in $S$ generating a subspace of $\PG(N,q)$ containing
$A$, and $\rho$ is the smallest value with this property.

\textbf{(ii)} Every
point $A\in\PG(N,q)$ (in homogeneous coordinates) can be written as a linear combination of at most $\rho+1$ points of $S$, and $\rho$ is the smallest value with this property (cf. Definition~\ref{Def1_CoverRad}(ii)).
\end{definition}

\begin{definition}\label{def1_minim} A   $\rho$-saturating set in $\PG(N,q)$  is \emph{minimal} if it does not contain a smaller $\rho$-saturating   set in $\PG(N,q)$.
\end{definition}

Saturating sets are considered  in \cite{BDGMP-R2R3CC_2019,BrPlWi,Dav95,Handbook-coverings,Janwa,%
DFMP-IEEE-LO,DGMP_ACCT2008,DGMP-AMC,DMP-JCTA2003,DavOstEJC,DavOst-IEEE2001,EtzStorm2016,Giul2013Survey,Klein-Stor,LandSt,Ughi}. In the literature, saturating sets are also called ``saturated
sets'', ``spanning sets'', ``dense sets''.

Let $s_q(N,\rho)$ be \emph{the smallest size of a $\rho$-saturating set} in $\mathrm{PG}(N,q)$.

If a column of an $r\times n$ parity check matrix of an $[n,n-r]_qR$ code is treated as a point (in homogeneous coordinates) of $\PG(r-1,q)$ then this parity check matrix defines an $(R-1)$-saturating $n$-set in $\PG(r-1,q)$ \cite{BrPlWi,Dav95,DavOstEJC,Giul2013Survey,DGMP-AMC,EtzStorm2016,Klein-Stor,LandSt,Janwa}. There is a \emph{one-to-one correspondence between $[n,n-r]_qR$ codes and $(R-1)$-saturating $n$-sets in $\mathrm{PG}(r-1,q)$}. Therefore,
$\ell_q(r,R)=s_q(r-1,R-1).$ If the $[n,n-r]_qR$ code  is locally optimal then the corresponding $(R-1)$-saturating $n$-set is minimal.

The results of Proposition \ref{prop1_R>=4} are based on the so-called direct sum \cite[Sect.\,4.2]{DGMP-AMC} of codes with radius $R=2,3$ which use the following geometrical constructions:\\
$\bullet$ ``oval plus line'' \cite[p.\,104]{BrPlWi}, \cite[Th.\,3.1]{DavParis}, \cite[Th.\,5.1]{Dav95}; the construction gives a 1-saturating  $(2q+1)$-set in $\PG(3,q)$ corresponding to a $[2q+1,2q+1-4,3]_q2$ code with $r=4=2R$;\\
$\bullet$ ``two ovals plus line'' \cite[Sect.\,4]{DavOstEJC}; the construction gives a 2-saturating  $(3q+1)$-set in $\PG(5,q)$ that corresponds to a $[3q+1,3q+1-6,3]_q3$ code with $r=6=2R$.

\begin{problem}\label{probl2}
   For all $\rho\ge3$, obtain a construction of a $\rho$-saturating $((\rho+1)q+1)$-set in\linebreak
    $\PG(2\rho+1,q)$ that corresponds to an $[Rq+1,Rq+1-2R]_qR$ code with $R=\rho+1$; thereby prove that $s_q(2\rho+1,\rho)\le (\rho+1)q+1$ and $\ell_q(2R,R)\le Rq+1$.
\end{problem}

 Note that for $n < Rq+1$, no examples of  $[n,n-2R]_qR$ codes  seem to be known. Moreover, in \cite[Prop.\,4.2]{DGMP-AMC}, it is proved that $\ell_4(4,2)=s_4(3,1)=2\cdot4+1$.

\begin{problem}\label{probl3}
 \emph{\cite[Sects.\,4,\,5]{DGMP-AMC}} Determine whether  $\ell_q(2R,R)=Rq+1$.
\end{problem}

The results of Proposition \ref{prop1_evenR} are based on 1-saturating sets in the plane $\PG(2,q^2)$.
\begin{problem}\label{probl4}
In the projective plane $\PG(2,q)$ with $q$ square,  construct new $1$-saturating sets with sizes smaller than the known ones.
\end{problem}

The paper is organized as follows. In Sect. \ref{sec_main_res}, we summarize the main results of the paper. In Sect. \ref{sec_constr}, we propose a construction ``Line+Ovals'' for $\rho$-saturating sets in $\PG(2\rho+1,q)$ and codes of codimension $2R$. This solves Problem \ref{probl2}. In Sect.~\ref{sec4_qm_concat}, we give two constructions for code codimension lifting.  In Sect. \ref{sec5_inf_fam},  we
use the codes of Sect. \ref{sec_constr} as starting ones for the constructions of Sect. \ref{sec4_qm_concat} and obtain new infinite code families with fixed radius $R\ge4$ and codimension $tR$, $t\ge2$. This solves Problem~\ref{probl1}(i) for the most part. In Sect. \ref{sec_1sat_plane}, using the recent known results on double blocking sets, we obtain new 1-saturating sets in $\PG(2,q^2)$ that solves in part Problem~\ref{probl4}. Then starting from these sets, we obtain new
infinite code families with fixed even radii $R\ge2$ and codimension $tR+\frac{R}{2}$, $t\ge1$. This solves in part Problem~\ref{probl1}(ii).

\section{The main results}\label{sec_main_res}
The main results of this paper are as follows:

$\bullet$ Problem \ref{probl2} is solved, see Sect. \ref{sec_constr} where minimal $\rho$-saturating $((\rho+1)q+1)$-sets in $\PG(2\rho+1,q)$ are constructed. The minimality of these sets gives credence that  Problem \ref{probl3}  can be solved.

$\bullet$ Problem \ref{probl1}(i) is solved for the most part, see Sects. \ref{sec4_qm_concat} and \ref{sec5_inf_fam}. New constructive upper bounds based on Theorems \ref{th3_main_geom}, \ref{th3_main_codes}, \ref{th5_odd}, \ref{th5_even} are collected in Theorem \ref{th2_main}.

\begin{theorem}\label{th2_main}
\looseness=-1 For the length function $\ell_q(r,R)$ and for the smallest size $s_q(r-1,R-1)$ of an $(R-1)$-saturating set
in  $\PG(r-1,q)$ the following constructive bounds  hold:
  \begin{align*}
&\ell_q(r,R)=s_q(r-1,R-1)\le Rq^{(r-R)/R}+q^{(r-2R)/R}+\Delta_q(r,R),~r=tR,\displaybreak[3]\\
&\text{where for }m_1=\lceil\log_q (R+1)\rceil+1\text{ we have}\displaybreak[3]\\
    &\textbf{\emph{(i)}} ~~~\Delta_q(r,R)=0\text{ if }~t=2,~q=4\text{ and }q\ge7,~R\ge4;\displaybreak[3]\\
    &\textbf{\emph{(ii)}} ~~\Delta_q(r,R)=0\text{ if }~t=2,~q=5,~R=4,5;\displaybreak[3]\\
    &\textbf{\emph{(iii)}} ~\Delta_q(r,R)=0\text{ if }~t\ge\lceil\log_q R\rceil+3,~q\ge7 \text{ odd},~ R\ge4;\displaybreak[3]\\
    &\textbf{\emph{(iv)}} ~\,\Delta_q(r,R)=\sum_{j=2}^{t}q^{(r-jR)/R}~\text{ if }~m_1+2<t< 3m_1+2,~q\ge8\text{ even},~R\ge4;\displaybreak[3]\\
    &\textbf{\emph{(v)}} ~~\Delta_q(r,R)=\sum_{j=2}^{m_1+2}q^{(r-jR)/R}~\text{ if }~t=m_1+2 \text{ and }t\ge3m_1+2,~q\ge8\text{ even},~R\ge4.
  \end{align*}
\end{theorem}

The new bounds of Theorem \ref{th2_main} are better than the known ones of
 Proposition \ref{prop1_R>=4} where the coefficient for $q^{(r-2R)/R}$ is $\left\lceil \frac{R}{3}\right\rceil$ whereas in Theorem~\ref{th2_main} it is equal to 1 or 2.

$\bullet$ Problem \ref{probl4} is solved in part, see Sect. \ref{sec_1sat_plane}. We use the following notation:
\begin{align}\label{eq2_phi}
&\phi(q) \text{ is the order of the largest proper subfield of }\F_{q};\\
& f_q(r,R)=\left\{
 \begin{array}{ccc}
   0 & \text{ if }&r\ne\frac{9R}{2},\frac{13R}{2}\medskip \\
  q^{(r-3R)/R-0.5}+q^{(r-4R)/R-0.5} & \text{ if }& r=\frac{9R}{2},\frac{13R}{2}
 \end{array}
 \right..\label{eq2_fqrR}
\end{align}
\looseness=-1 By Proposition \ref{prop6_1sat}(v),(vi),   in $\PG(2,q)$, $q=p^{2h}$, $h\ge2$, there are  $1$-saturating $n$-sets with
\begin{align*}
n=2\sqrt{q}+2\frac{\sqrt{q}-1}{\phi(\sqrt{q}\,)-1},~p\ge3\text{ prime};~~n=2\sqrt{q}+2\frac{\sqrt{q}}{p}+2,~p\ge7\text{ prime}.
 \end{align*}
These new 1-saturating sets have smaller sizes than the known ones, see Remark \ref{rem6_improve}.

$\bullet$ Problem \ref{probl1}(ii) is solved in part. New bounds based on Theorem \ref{th6_evenR_r=tR+R/2} are as follows.
\begin{theorem}\label{th2_evenR_r=tR+R/2}
Let $R\ge2$ be even. Let $p$ be prime, $q=p^{2\eta},~\eta\ge2$, $r=tR+\frac{R}{2}$, $t\ge1$.\\  The following constructive upper bounds on the length function hold:
\begin{align*}
&\textbf{\emph{(i)}} ~\ell_q(r,R)\le R\left(1+\frac{\sqrt{q}-1}{\sqrt{q}(\phi(\sqrt{q}\,)-1)}\right)q^\frac{r-R}{R}+R\left\lfloor  q^{(r-2R)/R-0.5}\right\rfloor+\frac{R}{2}f_q(r,R),\,p\ge3;\displaybreak[3]\\
&\textbf{\emph{(ii)}} ~\ell_q(r,R)\le R\left(1+\frac{1}{p}+\frac{1}{\sqrt{q}}\right)q^{(r-R)/R}+R\left\lfloor q^{(r-2R)/R-0.5}\right\rfloor+\frac{R}{2}f_q(r,R),~p\ge7.
\end{align*}
  \end{theorem}

If $\sqrt{q}=p^\eta$ with $\eta\ge3$ odd,  the new bounds of Theorem \ref{th2_evenR_r=tR+R/2} are better than the known ones of Proposition \ref{prop1_evenR}. If e.g. $q=p^6$, $\eta=3$, then the bound of Theorem~\ref{th2_evenR_r=tR+R/2}(ii) is by $Rq^{(r-R)/R-1/3}$ smaller than
the known one of \eqref{eq1_h=3}. Also, the new bound holds for all $p\ge7$ whereas in  \eqref{eq1_h=3} $p\le73$. Moreover, if $\eta\ge5$ odd, the known bounds  \eqref{eq1_h=1} have the main term  $\frac{3}{2}Rq^{(r-R)/R}$ whereas for the new bounds it is   $Rq^{(r-R)/R}$.
%\end{remark}

\section{Construction ``Line+Ovals'' for $\rho$-saturating sets in\\ $\PG(2\rho+1,q)$ and codes of codimension $2R$}\label{sec_constr}
\looseness=-1\textbf{Notation.} Throughout the paper we denote by $x_i$, $i=0,1,\ldots,N$, homogeneous coordinates of points of $PG(N,q)$. In the other words, a point $(x_0x_1\ldots x_N)\in\PG(N,q)$. The leftmost nonzero  coordinate is equal to~1. In general, by default, $x_i\in\F_q$. If $x_i\in\F_q^*$, we denote it as~$\x_i$. If $(x_i\ldots x_{i+m})\ne(0\ldots0)$, we denote it as $\overline{x_i\ldots x_{i+m}}$. Also, we  write explicit values 0,1 for some coordinates or denote coordinates by the letters $a,a_j$ that are elements of $\F_q$.

\subsection{The construction}
Let $\F_q=\{a_1=0,a_2,\ldots,a_q\}$ be the Galois field of order $q$. Let $\F_q^*=\F_q\setminus\{0\}=\{a_2,\ldots,a_q\}$.
Denote $\Sigma_\rho=\PG(2\rho+1,q)$.
Let $\Sigma_u$ be the $(2u+1)$-dimensional projective subspace of $\Sigma_\rho$ such that
\begin{align*}
  \Sigma_u=\{(\underbrace{x_0x_1\ldots x_{2u+1}}_{2u+2}\underbrace{0\ldots0}_{2\rho-2u}):x_i\in\F_q\}\subseteq\Sigma_\rho,
~u=0,1,\ldots,\rho.
\end{align*}
In $\Sigma_u$, let $\pi_u$ be the plane such that
\begin{align*}
 \pi_u=\{(\underbrace{0\ldots0}_{2u-1}x_{2u-1}x_{2u}x_{2u+1}\underbrace{0\ldots0}_{2\rho-2u}):x_i\in\F_q\}\subset\Sigma_u,\,u=1,2,\ldots,\rho.
\end{align*}
In $\pi_u$, let $A_u^0$ and $A_u^\infty$ be the points of the form
\begin{align*}
&  A_u^0=(\underbrace{0\ldots0}_{2u-1}100\underbrace{0\ldots0}_{2\rho-2u})\in\pi_u,~~ A_u^\infty=(\underbrace{0\ldots0}_{2u-1}001\underbrace{0\ldots0}_{2\rho-2u})\in\pi_u,~u=1,2,\ldots,\rho.
\end{align*}
In $\pi_u$, let $C_u$ and $C_u^*$ be the conic and the truncated one, respectively, of the form
\begin{align*}
C_u=C_u^*\cup\{A^0_u,A^\infty_u\},~~C_u^*=\{(\underbrace{0\ldots0}_{2u-1}1aa^2\underbrace{0\ldots0}_{2\rho-2u}):a\in\F_q^*\},~u=1,2,\ldots,\rho.
\end{align*}
Let $T_u$ be the nucleus of $C_u$, if $q$ is even, or the intersection of the tangents to  $C_u$ in the points $A_u^0$ and $A_u^\infty$, if $q$ is odd, so that
$T_u=(\underbrace{0\ldots0}_{2u-1}010\underbrace{0\ldots0}_{2\rho-2u})\in\pi_u,~ u=1,2,\ldots,\rho.$

\noindent In $\Sigma_0$, let $A_0^0$ and $A_0^\infty$ be the points of the form $A_0^0=(10\underbrace{0\ldots0}_{2\rho}),~A_0^\infty=(01\underbrace{0\ldots0}_{2\rho})$.
Also,  let $L_0$ and $L^*_0$ be the line and the truncated one, respectively, such that
\begin{align*}%\label{eq3_L}
L_0=L^*_0\cup\{A_0^0,A_0^\infty\}\subset\Sigma_0,~ L^*_0=\{(1a\underbrace{0\ldots0}_{2\rho}):a\in\F_q^*\}\subset\Sigma_0.
\end{align*}

Note that by Definition \ref{def1_usual satur}, a 0-saturating set in $PG(N,q)$ is the whole space.

\noindent\textbf{Construction S. (``Line+Ovals'')}
Let $\rho\ge0$. Let $S_\rho=\{P_1,P_2,\ldots,P_{(\rho+1)q+1}\}$ be a point\linebreak
 $((\rho+1)q+1)$-subset of $\Sigma_\rho=\PG(2\rho+1,q)$. Let $P_j$ be the $j$-th point of $S_\rho$.
We construct $S_\rho$ as follows:
\begin{align}
\label{eq3_constr}
&S_0=\{A_0^0\}\cup L_0^*\cup\{A^\infty_0\}=\left\{P_1,P_2,\ldots,P_{q+1}\right\}=\Sigma_0=\PG(1,q);\displaybreak[3]\\
& S_{\rho}=\{A_0^0\}\cup L^*_0\cup\bigcup_{u=1}^{\rho}\left(C_u^*\cup\{T_u\}\right)\cup\{A^\infty_\rho\}
=\left\{P_1,P_2,\ldots,P_{(\rho+1)q+1}\right\}\subset\Sigma_\rho\text{ if }\rho\ge1.\displaybreak[3]\notag\\
& P_1=(10\underbrace{0\ldots0}_{2\rho})=A_0^0;~~P_j=(1a_j\underbrace{0\ldots0}_{2\rho}),~a_j\in\F_q^*,~~j=2,3,\ldots,q.\displaybreak[3]\\
& P_{uq+j-1}=(\underbrace{0\ldots0}_{2u-1}1a_ja_j^2\underbrace{0\ldots0}_{2\rho-2u}),~a_j\in\F_q^*,~~u=1,2,\ldots,\rho,~j=2,3,\ldots,q.\displaybreak[3]\\
\label{eq3_points1}
  & P_{(u+1)q}=(\underbrace{0\ldots0}_{2u-1}010\underbrace{0\ldots0}_{2\rho-2u})=T_{u},~u=1,2,\ldots,\rho;~~P_{(\rho+1)q+1}=A^\infty_\rho.
\end{align}
Also, the set $S_\rho$ can be represented in the matrix form $\widehat{\Hh}_\rho$, where every column is a point in homogeneous coordinates. We have
\begin{align}\label{eq3_points}
&S_\rho=\widehat{\Hh}_\rho\\
&=\left[
 \begin{array}{@{}c@{\,}||@{\,}c@{\,}c@{\,}c@{\,}||@{\,\,}c@{\,}c@{\,}c@{\,}|@{\,}c@{}||
 @{\,\,}c@{\,}c@{\,}c@{\,}|@{\,}c@{\,}||@{\,}c@{\,}||
 @{\,}c@{\,}c@{\,}c@{\,}|@{\,}c@{}||@{\,\,}c@{\,}c@{\,}c@{\,}|@{\,}c@{\,}||@{\,}c@{} }
1&   1&\ldots&     1&0    &\ldots&0    &0&0    &\ldots&0    &0&\ldots&0&\ldots&0&0&0&\ldots&0&0&0 \\
0& a_2&\ldots& a_q  &1    &\ldots&1    &0&0    &\ldots&0    &0&\ldots&0&\ldots&0&0&0&\ldots&0&0&0 \\
0&0&\ldots&     0&a_2  &\ldots&a_q  &1&0    &\ldots&0    &0&\ldots&0&\ldots&0&0&0&\ldots&0&0&0 \\
0&0&\ldots&     0&a_2^2&\ldots&a_q^2&0&1    &\ldots&1    &0&\ldots&0&\ldots&0&0&0&\ldots&0&0&0 \\
0&0&\ldots&     0&0    &\ldots&0    &0&a_2  &\ldots&a_q  &1&\ldots&0&\ldots&0&0&0&\ldots&0&0&0 \\
0&0&\ldots&     0&0    &\ldots&0    &0&a_2^2&\ldots&a_q^2&0&\ldots&0&\ldots&0&0&0&\ldots&0&0&0\\
    & &\ldots&      &     &\ldots&     & &     &\ldots&     & &\ldots& &\ldots& & & &\ldots& &\\
0&0&\ldots&     0&0    &\ldots&0    &0&0    &\ldots&0    &0&\ldots&1    &\ldots&1    &0&0    &\ldots&0    &0&0\\
   0&0&\ldots&     0&0    &\ldots&0    &0&0    &\ldots&0    &0&\ldots&a_2  &\ldots&a_q  &1&0    &\ldots&0    &0&0\\
   0&0&\ldots&     0&0    &\ldots&0    &0&0    &\ldots&0    &0&\ldots&a_2^2&\ldots&a_q^2&0&1    &\ldots&1    &0&0\\
   0&0&\ldots&     0&0    &\ldots&0    &0&0    &\ldots&0    &0&\ldots&0    &\ldots&0    &0&a_2  &\ldots&a_q  &1&0\\
   0&0&\ldots&     0&0    &\ldots&0    &0&0    &\ldots&0    &0&\ldots&0    &\ldots&0    &0&a_2^2&\ldots&a_q^2&0&1\\
   -&-&-&-&-&-&-&-&-&-&-&-&-&-&-&-&-&-&-&-&-&-\\
 A_0^0&   &L_0^*  &      &     &C_1^*&   &T_1&     &C_2^*      &&T_2&\ldots&  &C_{\rho-1}^*  &&T_{\rho-1}&&C_{\rho}^*&&T_{\rho}& A^{\infty}_{\rho}
 \end{array}
\right].\notag
\end{align}
\begin{remark}
The sets $S_1$ and $S_2$ of Construction S are, respectively, the $1$-saturating set in $\PG(3,q)$ of the construction ``oval plus line'' \cite[p.\,104]{BrPlWi}, \cite[Th.\,3.1]{DavParis}, \cite[Th.\,5.1]{Dav95} and the $2$-saturating set in $\PG(5,q)$ of the construction ``two ovals plus line'' \cite[Sect.\,4]{DavOstEJC}.
\end{remark}
\subsection{Saturation of Construction S}
We say that a point $A\in\PG(N,q)$ is \emph{$\rho$-covered} by a set $S\subseteq\PG(N,q)$ if $A$ is a linear combination of  less than or equal to
 $\rho+1$ points of $S$. A subset $G\subset\PG(N,q)$ is \emph{$\rho$-covered} by $S$ if all points of $G$ are $\rho$-covered by $S$.

\begin{definition}\label{def3_subst} Let $S$ be a $\rho$-saturating set in $\PG(N,q)$.
A point $A\in S$ is \emph{$\rho$-essential} if $S\setminus \{A\}$ is no longer a $\rho$-saturating set. A point $A\in S$ is \emph{$\rho$-essential} for a set $\widetilde{\M_\rho}(A)\subset\PG(N,q)$ if all points of  $\widetilde{\M_\rho}(A)$ are not $\rho$-covered by $S\setminus \{A\}$. We denote by $\M_\rho(A)$ a set such that $\widetilde{\M_\rho}(A)\subseteq\M_\rho(A)\subset\PG(N,q)$.
\end{definition}

The following proposition and lemma are obvious.

\begin{proposition}\label{prop3_S0}
  Let $q\ge3$. Let $\Sigma_0=\PG(1,q)$. Let the set $S_0=\{A_0^0\}\cup L_0^*\cup\{A^\infty_0\}\subset\Sigma_0$ be as in \eqref{eq3_constr}--\eqref{eq3_points}. Then it holds that

\emph{\textbf{(i)}} The $(q+1)$-set $S_0$ is a minimal $0$-saturating set in $\Sigma_0$.

\emph{\textbf{(ii)}} The point $A^\infty_0$ of  $S_0$ is $0$-essential for the set $\widetilde{\M_0}(A^\infty_0)$ such that
\begin{align}\label{eq3_M_0}
 \widetilde{\M_0}(A^\infty_0)
 =\M_0(A^\infty_0)=\{A^\infty_0\}=\{(01)\}.
\end{align}

\emph{\textbf{(iii)}} The $q$-set $S_0\setminus\{A^\infty_0\}$ is $1$-saturating in $\Sigma_0$.
\end{proposition}

\begin{lemma}\label{lem3_Ainf=A0__pi_u}
 Let $q\ge4$, $\rho\ge2$. Then the plane $\pi_u$, $u=1,\ldots,\rho$, is $2$-covered by $C_u^*$.
 Also, the point $A^\infty_u= A^0_{u+1}$, $u=1,\ldots,\rho-1$, is $2$-covered by $C_u^*$ as well as by $C_{u+1}^*$.
\end{lemma}

\begin{lemma}\label{lem3_piu-A0,Ainf_pirho-A0}
 Let $q=4$ or $q\ge7$. Then all points of $\pi_u\setminus\{A^0_u,A^\infty_u\}$ are $1$-covered by $C_u^*\cup\{T_u\}$, $u=1,\ldots,\rho$.
Also, all points of $\pi_\rho\setminus\{A^0_\rho\}$ are $1$-covered by $C_\rho^*\cup \{T_\rho ,A^\infty_\rho\}$.
\end{lemma}
\begin{proof}
If $q$ is even, every point of a plane outside of a hyperoval $C_u\cup \{T_u\}$ lies on $(q+2)/2$ its bisecants. If $q$ is odd, every point of a plane outside of a conic $C_u$ lies on at least $(q-1)/2$ its bisecants. At most two of these bisecants will be removed if one removes $A^0_u$ and $A^\infty_u$
 from~$C_u$.  Thus, for $q=4$ and $q\ge7$, every point of $\pi_u\setminus\{A^0_u,A^\infty_u\}$  lies on at least one bisecant of $C_u^*\cup \{T_u\}$.
The same holds for $\pi_\rho\setminus\{A^0_\rho\}$.
\end{proof}

\begin{proposition}\label{prop3_S1}
Let $q=4$ or $q\ge7$. Let $\Sigma_1=\PG(3,q)$. Let the set $S_1=\{A_0^0\}\cup L_0^*\cup C_1^*\cup \{T_1,A^\infty_1\}\subset\Sigma_1$ be as in~\eqref{eq3_constr}--\eqref{eq3_points}.  Let $\M_0(A^\infty_0)$ be as in \eqref{eq3_M_0}.  Then it holds that

  \emph{\textbf{(i)}} The $(2q+1)$-set $S_1$ is a minimal $1$-saturating set in $\Sigma_1$.

  \emph{\textbf{(ii)}} The point $A^\infty_1$ of $S_1$ is $1$-essential for the set $\widetilde{\M_1}(A^\infty_1)$ such that
\begin{align}\label{eq3_MAinfS1}
\widetilde{\M_1}(A^\infty_1)=\M_1(A^\infty_1)=\{(x_0\ldots x_3):(x_0x_1)\notin\M_0(A^\infty_0),(x_2x_3)=(0\x_3)\}.
\end{align}

  \emph{\textbf{(iii)}} The $2q$-set $S_1\setminus\{A^\infty_1\}$ is $2$-saturating in $\Sigma_1$.
\end{proposition}
\begin{proof} \textbf{(i)}
  By Proposition \ref{prop3_S0}(iii) and Lemma \ref{lem3_piu-A0,Ainf_pirho-A0}, $\Sigma_0$   and $\pi_1$  are $1$-covered by $\{A_0^0\}\cup L_0^*\cup C_1^*\cup \{T_1,A^\infty_1\}$. Hence, we should consider points of the form
  \begin{align}\label{eq3_BS1}
    B=(\widehat{x}_0x_1\overline{x_2x_3})=(1x_1\overline{x_2x_3})\in\Sigma_1\setminus(\Sigma_0\cup\pi_1).
  \end{align}
 We show that $B$ in \eqref{eq3_BS1} is a linear combination of at most 2 points of $S_1$.

1) Let $(x_0x_1)\in\M_0(A^\infty_0)$. By \eqref{eq3_BS1}, we have no such points $B$.

2) Let $(x_0x_1)\notin\M_0(A^\infty_0)$.
By the hypothesis, $(x_0x_100)$ is 0-covered by $S_0\setminus\{A^\infty_0\}$, i.e. $(x_0x_100)=(1x_100)\in\{A_0^0\}\cup L_0^*$.  For $B$ of  \eqref{eq3_BS1},    we have
  \begin{align}\label{eq3_Ainf_rhoS1}
  & B=(x_0x_10\x_3)=(x_0x_100)+\x_3(0001)=(x_0x_100)+\x_3A^\infty_1;\\
  &B=(x_0x_1\x_20)=(x_0x_100)+\x_2(0010)=(x_0x_100)+\x_2T_1;\notag\\
  &B=(x_0x_1\x_2\x_3)=(x_0z00)+\frac{\x_2^2}{\x_3}(01yy^2),~z=x_1-\frac{\x_2^2}{\x_3},~y=\frac{\x_3}{\x_2}.\notag
  \end{align}
Note that $(x_0z00)=(1z00)$ is 0-covered by   $S_0\setminus\{A^\infty_0\}$ for any $z$.

 From \eqref{eq3_Ainf_rhoS1}, we see that all points of $S_1$ are 1-essential.

  \textbf{(ii)}
The assertion follows from \eqref{eq3_Ainf_rhoS1}.

\textbf{(iii)}  We have, cf. \eqref{eq3_Ainf_rhoS1}, $(1x_10\x_3)=(1z00)+(010\x_3)$, where $z=x_1-1$ and\linebreak
 $(010\x_3)\in\pi_1\setminus\{A_1^0, A^\infty_1\}$ is 1-covered by $C_1^*\cup\{T_1\}$, see Lemma~\ref{lem3_piu-A0,Ainf_pirho-A0}.
               \end{proof}

\begin{proposition}\label{prop3_S2}
Let $q=4$ or $q\ge7$. Let $\Sigma_2=\PG(5,q)$. Let the set $S_2=\{A_0^0\}\cup L_0^*\cup C_1^*\cup \{T_1\}\cup C_2^*\cup \{T_2,A^\infty_2\}\subset\Sigma_2$ be as in~\eqref{eq3_constr}--\eqref{eq3_points}. Let $\M_1(A^\infty_1)$ be as in \eqref{eq3_MAinfS1}.  Then it holds that

  \emph{\textbf{(i)}}
The $(3q+1)$-set $S_2$  is a minimal $2$-saturating set in $\Sigma_2$.

\emph{\textbf{(ii)}}  The point $A^\infty_2$ of  $S_2$ is $2$-essential for the set $\widetilde{\M_2}(A^\infty_2)$ such that
  \begin{align}\label{eq3_MAinfS2}
  \widetilde{\M_2}(A^\infty_2)\subset\M_2(A^\infty_2)=\{(x_0\ldots x_5):(x_0\ldots x_3)\notin\M_1(A^\infty_1),~(x_4x_5)=(0\x_5)\}.
  \end{align}

\emph{\textbf{(iii)}}
The $3q$-set $S_2\setminus\{A^\infty_2\}$  is $3$-saturating in $\Sigma_2$.

\end{proposition}
\begin{proof} \textbf{(i)}
By Propositions \ref{prop3_S0} and \ref{prop3_S1} and Lemmas \ref{lem3_Ainf=A0__pi_u} and \ref{lem3_piu-A0,Ainf_pirho-A0}, it holds that  $\Sigma_0$ is $1$-covered by $\{A_0^0\}\cup L_0^*$;
$\pi_1$  and $\pi_2$  are 2-covered by $C_1^*$ and $C_2^*$, respectively; $\pi_2\setminus\{A^0_2\}$ is $1$-covered by $C_2^*\cup \{T_2,A^\infty_2\}$; $\Sigma_1$  is 2-covered by $S_1\setminus\{A^\infty_1\}$. Recall that $\Sigma_0\cup\pi_1\subset\Sigma_1$. So, we should consider points of the form
\begin{align}\label{eq3_BS2}
B=(\overline{x_0x_1x_2}x_3\overline{x_4x_5})\in\Sigma_2\setminus(\Sigma_1\cup\pi_2).
\end{align}
 We show that $B$ in \eqref{eq3_BS2} is a linear combination of at most 3 points of $S_2$.

1) Let $(x_0\ldots x_3)\in\M_1(A^\infty_1)$.
By the hypothesis and by \eqref{eq3_MAinfS1},  \eqref{eq3_BS2}, we have
\begin{align*}
(x_0x_1)\notin\M_0(A^\infty_0),~  B=(x_0x_10\x_3\overline{x_4x_5})=(x_0x_10000)+(000\x_3\overline{x_4x_5}),
\end{align*}
where $(x_0x_10000)$ is 0-covered by  $S_0\setminus\{A^\infty_0\}$ and $(000\x_3\overline{x_4x_5})\in\pi_2\setminus\{A^0_2,A^\infty_2\}$ is 1-covered by $C_2^*\cup\{T_2\}$, see Lemma \ref{lem3_piu-A0,Ainf_pirho-A0}.

2) Let $(x_0\ldots x_3)\notin\M_1(A^\infty_1)$.

By the hypothesis, $(x_0\ldots x_300)$ is 1-covered by $S_1\setminus\{A^\infty_1\}$.  Also,
 \begin{align}\label{eq3_Ainf_rhoS2}
&  B=(x_0\ldots x_30\x_5)=(x_0\ldots x_300)+\x_5(000001)=(x_0\ldots x_300)+\x_5A^\infty_2;\displaybreak[3]\\
&B=(x_0\ldots x_3\x_40)=(x_0\ldots x_300)+\x_4(000010)=(x_0\ldots x_300)+\x_4T_2;\displaybreak[3]\label{eq3_T_rhoS2}\\
&B=(x_0\ldots x_3\x_4\x_5)=(x_0x_1x_2z00)+\frac{\x_4^2}{\x_5}(0001yy^2),~z=x_3-\frac{\x_4^2}{\x_5},~y=\frac{\x_5}{\x_4}.\label{eq3_v4v5S2}
  \end{align}
In \eqref{eq3_Ainf_rhoS2}, \eqref{eq3_T_rhoS2}, $B$ is a linear combination of at most $(1+1)+1=3$ points. If $(x_0x_1x_2z)\notin\M_1(A^\infty_1)$, then the representation \eqref{eq3_v4v5S2} is the needed linear combination. If $(x_0x_1x_2z)\in\M_1(A^\infty_1)$  whereas $(x_0\ldots x_3)\notin\M_1(A^\infty_1)$, then the only possible case is   $(x_0x_1)\notin\M_0(A^\infty_0)$ with $(x_2x_3)=(00)$, see
\eqref{eq3_MAinfS1}. In this case,
  \begin{align}\label{eq3_Ainf_rhoS2_1x_100}
&B=(x_0x_100\x_4\x_5)=(1x_100\x_4\x_5)=(1x_10000)+(0000\x_4\x_5),
  \end{align}
  where $(1x_10000)$ is 0-covered by $\{A_0^0\}\cup L_0^*$ and $(0000\x_4\x_5)\in\pi_2\setminus\{A^0_2,A^\infty_2\}$ is $1$-covered by $C_2^*\cup\{T_2\}$, see Lemma \ref{lem3_piu-A0,Ainf_pirho-A0}. Thus, $B$ in \eqref{eq3_Ainf_rhoS2_1x_100} is a linear combination of at most $(0+1)+(1+1)=3$ points.

  From \eqref{eq3_Ainf_rhoS2}--\eqref{eq3_Ainf_rhoS2_1x_100} we see that all points of $S_2\setminus S_1$ are 2-essential. Also, we take into account that $S_1$ is a \emph{minimal} 1-saturating set.

\textbf{(ii)}
The assertion follows from  \eqref{eq3_Ainf_rhoS2}.
For some (but not for all) points in \eqref{eq3_Ainf_rhoS2} we could avoid use of $A^\infty_2$; this explains the sign ``$\subset$'' in \eqref{eq3_MAinfS2}. Let, for example, $B=(001\x_30\x_5)\notin\M_1(A^\infty_1)$. Then $B=(001000)+\x_3\left(00010\frac{\x_5}{\x_3}\right),$ where $(001000)=T_1$ and  $\left(00010\frac{\x_5}{\x_3}\right)\in\pi_2\setminus\{A^0_2,A^\infty_2\}$ is $1$-covered by $C_2^*\cup \{T_2\}$. But, if $B=(00100\x_5)\notin\M_1(A^\infty_1)$, we are not able to avoid~$A^\infty_2$.

\textbf{(iii)}   We have, cf.~\eqref{eq3_Ainf_rhoS2}, $ B=(x_0\ldots x_30\x_5)=(x_0x_1x_2z00)+(00010\x_5),$
where $z=x_3-1$ and $(00010\x_5)\in\pi_2\setminus\{ A_2^0,A^\infty_2\}$ is 1-covered by $C_2^*\cup\{T_2\}$, see Lemma~\ref{lem3_piu-A0,Ainf_pirho-A0}. This representation of $B$ is the needed linear combination of at most $(1+1)+(1+1)=4$ columns if $(x_0x_1x_2z)\notin\M_1(A^\infty_1)$ whence $(x_0x_1x_2z00)$ is 1-covered by $S_1\setminus \{A^\infty_1\}$.

But if $(x_0x_1x_2z)\in\M_1(A^\infty_1)$, then by \eqref{eq3_MAinfS1}, $(x_0x_1)\notin\M_0(A^\infty_0)$ and we have, similarly to~\eqref{eq3_Ainf_rhoS2_1x_100},
$B=(1x_1000\x_5)=(1x_10000)+\x_5(000001),$
 where
$(1x_10000)$ is 0-covered by $\{A_0^0\}\cup L_0^*$  and $(000001)=A^\infty_2\in\pi_2$ is 2-covered by $C_2^*$, see Lemma~\ref{lem3_Ainf=A0__pi_u}.
\end{proof}

%\subsection{Saturation of Construction S for any $\rho$}
\begin{theorem}\label{th3_main_geom}
Let $q=4$ or $q\ge7$. Let $\Upsilon\ge1$. Let $\Sigma_\rho=\PG(2\rho+1,q)$. Let $S_\rho$ be a point $((\rho+1)q+1)$-subset of $\Sigma_\rho$ as in Construction S of \eqref{eq3_constr}--\eqref{eq3_points}.  Then it holds that

  \emph{\textbf{(i)}}
The $((\rho+1)q+1)$-set $S_\rho$ is a minimal $\rho$-saturating set in $\Sigma_\rho$, $\rho=0,1,\ldots,\Upsilon$.

  \emph{\textbf{(ii)}} The point $A^\infty_\rho$ of  $S_\rho$ is $\rho$-essential for the set $\widetilde{\M_\rho}(A^\infty_\rho)$ such that
  \begin{align}
    &\widetilde{\M_0}(A^\infty_0)=\M_0(A^\infty_0)=\{(01)\},\notag\\
    &\widetilde{\M_1}(A^\infty_1)=\M_1(A^\infty_1)=\{(x_0\ldots x_3):(x_0x_1)\notin\M_0(A^\infty_0),(x_2x_3)=(0\x_3)\},\notag\\
    &\widetilde{\M_\rho}(A^\infty_\rho)\subset\M_\rho(A^\infty_\rho)=\{(x_0\ldots x_{2\rho+1}):(x_0\ldots x_{2\rho-1})\notin\M_{\rho-1}(A^\infty_{\rho-1}),\label{eq3_MAinfSrho} \\
    &(x_{2\rho}x_{2\rho+1})=(0\x_{2\rho+1})\},~\rho=2,3,\ldots,\Upsilon.\notag
  \end{align}

 \emph{\textbf{(iii)}}
The $(\rho+1)q$-set $S_\rho\setminus\{A^\infty_\rho\}$  is $(\rho+1)$-saturating in $\Sigma_\rho$, $\rho=0,1,\ldots,\Upsilon$.
\end{theorem}
\begin{proof}
  We prove by induction on $\Upsilon$.

  For $\Upsilon=3$ the theorem is proved in Propositions \ref{prop3_S0}, \ref{prop3_S1}, \ref{prop3_S2}.

  \emph{Assumption: let the assertions \emph{(i)--(iii)} hold for some $\Upsilon\ge3$.}

  We show that under Assumption, the assertions hold for $\Gamma=\Upsilon+1$.

  \textbf{(i)} By Propositions \ref{prop3_S0}, \ref{prop3_S1}, and \ref{prop3_S2}, Lemmas \ref{lem3_piu-A0,Ainf_pirho-A0} and \ref{lem3_Ainf=A0__pi_u}, and Assumption, we have the following:  $\Sigma_0$  is $1$-covered by $\{A_0^0\}\cup L_0^*$;
 $\pi_1\setminus\{ A^\infty_1\}$, $\pi_u\setminus\{A^0_u,A^\infty_u\}$, $u=2,3,\ldots,\Gamma$, are $1$-covered by $\{A_0^0\}\cup L_0^*\cup\bigcup\limits_{u=1}^{\Gamma}\left(C_u^*\cup\{T_u\}\right)$;  $\pi_{\Gamma}\setminus\{A^0_{\Gamma}\}$ is $1$-covered by $C_\Gamma^*\cup \{T_\Gamma,A^\infty_\Gamma\}$; $\pi_1$, $\pi_2$, $\ldots$, $\pi_\Gamma$ are 2-covered by $C_1^*$, $C_2^*$, $\ldots$,  $C_\Gamma^*$, respectively; $\Sigma_\Upsilon$ is $\Gamma$-covered by $S_\Upsilon\setminus\{A^\infty_\Upsilon\}$. Recall that $\Sigma_0\cup\bigcup\limits_{u=1}^\Upsilon\pi_u\subset\Sigma_\Upsilon$.
So, we should consider points of the form
\begin{align}\label{eq3_BS3}
  B=(\overline{x_0\ldots x_{2\Gamma-2}}x_{2\Gamma-1}\overline{x_{2\Gamma}x_{2\Gamma+1}})\in\Sigma_{\Gamma}\setminus(\Sigma_\Upsilon\cup\pi_{\Gamma}).
\end{align}
 We show that $B$ in \eqref{eq3_BS3} is a linear combination of at most $\Gamma+1$ points of $S_\Gamma$.

1) Let $(x_0\ldots x_{2\Gamma-1})\in\M_\Upsilon(A^\infty_\Upsilon)$.\\
By the hypothesis and by \eqref{eq3_MAinfSrho}, $(x_0\ldots x_{2\Upsilon-1})\notin\M_{\Upsilon-1}(A^\infty_{\Upsilon-1})$. Therefore,\\
 $(x_0\ldots x_{2\Upsilon-1}0000)$ is $(\Upsilon-1)$-covered by  $S_{\Upsilon-1}\setminus\{A^\infty_{\Upsilon-1}\}$.  Now by
\eqref{eq3_BS3}, we have
\begin{align}\label{eq3_Bin_rho}
&  B=(x_0\ldots x_{2\Upsilon-1}0\x_{2\Gamma-1}\overline{x_{2\Gamma}x_{2\Gamma+1}})=(x_0\ldots x_{2\Upsilon-1}0000)+(0\ldots0\x_{2\Gamma-1}\overline{x_{2\Gamma}x_{2\Gamma+1}}),
\end{align}
\looseness=-1 where $(0\ldots0\x_{2\Gamma-1}\overline{x_{2\Gamma}x_{2\Gamma+1}})\in\pi_\Gamma\setminus\{A^0_\Gamma,A^\infty_\Gamma\}$ is 1-covered by $C_\Gamma^*$, see Lemma \ref{lem3_piu-A0,Ainf_pirho-A0}. So, $B$ in \eqref{eq3_Bin_rho} is a linear combination of at most $(\Upsilon-1+1)+(1+1)=\Gamma+1$ points.

2)   Let $(x_0\ldots x_{2\Gamma-1})\notin\M_\Upsilon(A^\infty_\Upsilon)$. \\
By the hypothesis, $(x_0\ldots x_{2\Gamma-1}00)$ is $\Upsilon$-covered by $S_\Upsilon\setminus\{A^\infty_\Upsilon\}$. We can write
 \begin{align}\label{eq3_Ainf_rhoSrho}
&  B=(x_0\ldots x_{2\Gamma-1}0\x_{2\Gamma+1})=(x_0\ldots x_{2\Gamma-1}00)+\x_{2\Gamma+1}A^\infty_\Gamma;\displaybreak[2]\\
&B=(x_0\ldots x_{2\Gamma-1}\x_{2\Gamma}0)=(x_0\ldots x_{2\Gamma-1}00)+\x_{2\Gamma}T_\Gamma;\displaybreak[2]\label{eq3_T_rhoSrho}\\
&B=(x_0\ldots x_{2\Gamma-1}\x_{2\Gamma}\x_{2\Gamma+1})=(x_0\ldots x_{2\Gamma-2}z00)+\frac{\x_{2\Gamma}^2}{\x_{2\Gamma+1}}(0\ldots01yy^2),\label{eq3_v4v5Srho}\\
&z=x_{2\Gamma-1}-\frac{\x_{2\Gamma}^2}{\x_{2\Gamma+1}},~y=\frac{\x_{2\Gamma+1}}{\x_{2\Gamma}}.\notag
  \end{align}
In \eqref{eq3_Ainf_rhoSrho}, \eqref{eq3_T_rhoSrho}, $B$ is a linear combination of at most $(\Upsilon+1)+1=\Gamma+1$ points.
If\\ $(x_0\ldots x_{2\Gamma-2}z)\notin\M_\Upsilon(A^\infty_\Upsilon)$, then the representation \eqref{eq3_v4v5Srho} is the needed linear combination. If $(x_0\ldots x_{2\Gamma-2}z)\in\M_\Upsilon(A^\infty_\Upsilon)$  while $(x_0\ldots x_{2\Gamma-1})\notin\M_\Upsilon(A^\infty_\Upsilon)$, then the only possibility is\linebreak   $(x_0\ldots x_{2\Upsilon-1})\notin\M_{\Upsilon-1}(A^\infty_{\Upsilon-1})$ with $(x_{2\Gamma-2}x_{2\Gamma-1})=(00)$, see
\eqref{eq3_MAinfSrho}. In this case,
  \begin{align}\label{eq3_Ainf_rhoSrho_1x_100}
&B=(x_0\ldots x_{2\Upsilon-1}00\x_{2\Gamma}\x_{2\Gamma+1})=(x_0\ldots x_{2\Upsilon-1}0000)+(0\ldots0\x_{2\Gamma}\x_{2\Gamma+1}),
  \end{align}
  where $(x_0\ldots x_{2\Upsilon-1}0000)$ is $(\Upsilon-1)$-covered by  $S_{\Upsilon-1}\setminus\{A^\infty_{\Upsilon-1}\}$ and \\ $(0\ldots0\x_4\x_{2\Gamma-1})\in\pi_\Gamma\setminus\{A^0_\Gamma,A^\infty_\Gamma\}$ is $1$-covered by $C_\Gamma^*\cup\{T_\Gamma\}$, see Lemma \ref{lem3_piu-A0,Ainf_pirho-A0}. Thus, $B$ in \eqref{eq3_Ainf_rhoSrho_1x_100} is a linear combination of at most $(\Upsilon-1+1)+(1+1)=\Gamma+1$ points.

  From \eqref{eq3_Bin_rho}--\eqref{eq3_Ainf_rhoSrho_1x_100} we see that all the points of $S_\Gamma\setminus S_\Upsilon$ are $\Gamma$-essential. Also, we take into account that $S_\Upsilon$ is a \emph{minimal} $\Upsilon$-saturating set.

\textbf{(ii)}
The assertion \eqref{eq3_MAinfSrho} follows from  \eqref{eq3_Ainf_rhoSrho}.
For some (but not for all) points in \eqref{eq3_Ainf_rhoSrho} we could avoid use of $A^\infty_\Gamma$. This explains the sign ``$\subset$'' in \eqref{eq3_MAinfSrho}.

\textbf{(iii)}   We have, cf.~\eqref{eq3_Ainf_rhoSrho}, $B=(x_0\ldots x_{2\Gamma-1}0\x_{2\Gamma+1})=(x_0\ldots x_{2\Gamma-2}z00)+\\(0\ldots010\x_{2\Gamma+1}),$
   where $z=x_{2\Gamma-1}-1$ and $(0\ldots010\x_{2\Gamma+1})\in\pi_\Gamma\setminus\{ A_\Gamma^0,A^\infty_\Gamma\}$ is 1-covered by $C_\Gamma^*$, see Lemma~\ref{lem3_piu-A0,Ainf_pirho-A0}. This representation of $B$ is the needed linear combination of at most $(\Upsilon+1)+(1+1)=\Gamma+2$ points if  $(x_0\ldots x_{2\Gamma-2}z)\notin\M_\Upsilon(A^\infty_\Upsilon)$ whence $(x_0\ldots x_{2\Gamma-2}z00)$ is $\Upsilon$-covered by $S_\Upsilon\setminus A^\infty_\Upsilon$.

   But if  $(x_0\ldots x_{2\Gamma-2}z)\in\M_\Upsilon(A^\infty_\Upsilon)$, then by \eqref{eq3_MAinfSrho}, $(x_0\ldots x_{2\Upsilon-1}0000)\notin  \M_{\Upsilon-1}(A^\infty_{\Upsilon-1})$, and we have, cf. \eqref{eq3_Ainf_rhoSrho_1x_100},
   $(x_0\ldots x_{2\Upsilon-1}000\x_{2\Gamma+1})=(x_0\ldots x_{2\Upsilon-1}0000)+\x_{2\Gamma+1}(0\ldots01)$, where\linebreak $(x_0\ldots x_{2\Upsilon-1}0000)$ is $(\Upsilon-1)$-covered by  $S_{\Upsilon-1}\setminus\{A^\infty_{\Upsilon-1}\}$ and
  $(0\ldots01)=A^\infty_\Gamma\in\pi_\Gamma$ is 2-covered by $C_\Gamma^*$, see Lemma \ref{lem3_Ainf=A0__pi_u}.
\end{proof}

By computer search for $q=5$ we have proved the following proposition.
\begin{proposition}\label{prop3_q5}
  Let $q=5$. Let $0\le\rho\le4$. Let $\Sigma_\rho=\PG(2\rho+1,5)$. Let the $(5\rho+1)$-set $S_\rho\subset\Sigma_\rho$ be as in~\eqref{eq3_constr}--\eqref{eq3_points}. Then $S_\rho$  is a minimal $\rho$-saturating set in $\Sigma_\rho$.
\end{proposition}
\subsection{Codes of covering radius $R$ and codimension $2R$}
In the coding theory language, the results of this section give the following theorem.
\begin{theorem}\label{th3_main_codes}
Let $\widehat{V}_\rho$ be the code such that the columns of its parity check matrix are the points (in homogeneous coordinates) of
the $\rho$-saturating $((\rho+1)q+1)$-set $S_\rho$ of Construction S by \eqref{eq3_constr}--\eqref{eq3_points}.

\textbf{\emph{(i)}}  Let $q=4$ or $q\ge7$. Then for all $R\ge1$, the code $\widehat{V}_\rho$ is an $[Rq+1,Rq+1-2R,3]_qR$ locally optimal code
of covering radius $R=\rho+1$.

\textbf{\emph{(ii)}} Let $q=5$.  Then for $1\le R\le5$, the code $\widehat{V}_\rho$ is a $[5R+1,5R+1-2R,3]_5R$ locally optimal code
of covering radius $R=\rho+1$.
\end{theorem}
\begin{proof}
 We use Theorem \ref{th3_main_geom} and Proposition \ref{prop3_q5}. The code $\widehat{V}_\rho$ is locally optimal as the corresponding $\rho$-saturating set $S_\rho$ is minimal. Distance $d=3$ is due to $L^*_0$.
 \end{proof}
\begin{conjecture} Let $q=5$. Let $\widehat{V}_\rho$ be as in Theorem \ref{th3_main_codes}. Then for all $R\ge1$, the code $\widehat{V}_\rho$ is a $[5R+1,5R+1-2R,3]_5R$ locally optimal code
with radius $R=\rho+1$.
\end{conjecture}
\section{The $q^m$-concatenating constructions for code codimension lifting}\label{sec4_qm_concat}
The $q^{m}$-concatenating constructions are proposed in
\cite{DavPPI} and are developed in
\cite{DavParis,Dav95,DavCovRad2,DFMP-IEEE-LO,DavOst-IEEE2001,DGMP-AMC,DavOst-DESI2010},  see also
\cite {Handbook-coverings}, \cite[Sec.\thinspace 5.4]{CHLS-bookCovCod}.
By using a starting code as a ``seed'', a
$q^{m}$-concatenating construction yields an infinite family of
new codes
with a fixed covering radius, increasing codimension, and with  almost the same covering density.

We give versions of the $q^m$-concatenating constructions convenient for our goals. Several other versions of such constructions can be found in
\cite{DavPPI,DavParis,Dav95,DavCovRad2,DFMP-IEEE-LO,DavOst-IEEE2001,DGMP-AMC,DavOst-DESI2010}.
%\medskip

\noindent \textbf{Construction QM$_\mathbf{1}$.} Let columns $\h_{j}$ belong to $\F_{q}^{r_{0}}$ and let
$\Hh_{0}=[\mathbf{h}_{1} \h_{2}\mathbf{\ldots
h}_{n_{0}}]$ be a
parity check matrix of an $[n_{0},n_{0}-r_{0}]_{q}R,R$
\emph{starting} surface-covering code $V_{0}$ with $R\ge2$.  Let $m\geq 1$ be an integer
such that $q^m\ge n_0-1$. To each column
$\h_{j}$ we associate an element $\beta_{j}\in
\F_{q^{m}}\cup\{\ast\}$ so that $\beta _{i}\neq \beta _{j}$
if $i\ne j$.  Let a new code $V$ be
the $ [n,n-(r_{0}+Rm)]_{q}R_{V},\ell_V$ code with $n=q^{m}n_{0}$
and parity check matrix $\Hh_{V}$ of the  form
\begin{align}
&\Hh_{V} =\left[\B_{1}~\B_{2}~\ldots ~
\B_{n_{0}}\right] ,  \label{eq4_QM-H}\displaybreak[3]\\
&\B_{j} =\left[ \renewcommand{\arraystretch}{1.1}
\begin{array}{@{}cc@{}c@{}c@{}}
\h_{j} & \h_{j} & \mathbf{\cdots } & \h_{j} \\
\xi _{1} & \xi _{2} & \cdots & \xi _{q^{m}} \\
\beta _{j}\xi _{1} & \beta _{j}\xi _{2} & \cdots & \beta _{j}\xi _{q^{m}} \\
\beta _{j}^{2}\xi _{1} & \beta _{j}^{2}\xi _{2} & \cdots & \beta _{j}^{2}\xi
_{q^{m}} \\
\vdots & \vdots & \vdots & \vdots \\
\beta _{j}^{R-1}\xi _{1} & \beta _{j}^{R-1}\xi _{2} & \cdots & \beta
_{j}^{R-1}\xi _{q^{m}}
\end{array}
\right] \text{if }\beta_{j}\in
\F_{q^{m}},~~\mathbf{B}_{j}=\left[
\begin{array}{@{}cc@{}c@{}c@{}}
\mathbf{h}_{j} & \mathbf{h}_{j} & \mathbf{\cdots } & \mathbf{h}_{j} \\
0 & 0 & \cdots & 0 \\
\vdots & \vdots & \vdots & \vdots \\
0 & 0 & \cdots & 0 \\
\xi_{1} & \xi_{2} & \cdots & \xi_{q^{m}}
\end{array}
\right] \text{if }\beta_{j}=\ast ,  \label{eq4_Bj_QM_*}
\end{align}
where $\B_{j}$ is an $(r_{0}+Rm)\times q^m$ matrix, $0$ is the zero element of $\F_{q^{m}}$, $\xi_u$ is an element of $\F_{q^{m}}$, $\{\xi _{1},\xi _{2},\ldots ,\xi _{q^{m}}\}=\F_{q^{m}}$. An
element of $\F_{q^{m}}$ written in $\B_{j}$ denotes an
$m$-dimensional $q$-ary column vector that is a $q$-ary representation of this
element.

We denote $\bb_j(\xi_u)=(\h_j,\xi_u,\beta_j\xi_u,\beta_j^2\xi_u,\ldots,\beta_j^{R-1}\xi_u)$ the $u$-th column of $\B_j$ with $\beta_{j}\in
\F_{q^{m}}$. If $\beta _{j}=\ast$, we have $\bb_j(\xi_u)=(\h_j,0,\ldots,0,\xi_u)$.
\begin{theorem}\label{th4_QM1}
 In Construction QM$_1$, the new code $V$ with the parity check matrix \eqref{eq4_QM-H}, \eqref{eq4_Bj_QM_*} is an
$[n,n-(r_{0}+Rm),3]_{q}R,R$ surface-covering code with radius $R$ and length $n=q^{m}n_{0}$. If the starting code $V_0$ is locally optimal, then $V$  is locally optimal too.
\end{theorem}
\begin{proof}
 The minimum distance d is equal to 3
 since for any pair of columns $\bb_j(\xi_{u_1})$, $\bb_j(\xi_{u_2})$ of $\B_j$, a 3-rd one can be found such that the column triple corresponds to a codeword of weight~3. Take $a,b,c\in\F_q^*$ with $a+b+c=0$.  Put $\xi_{u_3}=(-a\xi_{u_1}-b\xi_{u_2})/c$. Let $\mathbf{0}$ be the zero $(r_0+Rm)$-positional column. Then for all $j$ we have
\begin{align}\label{eq4_3word}
 a\bb_j(\xi_{u_1})+b\bb_j(\xi_{u_2})+c\bb_j(\xi_{u_3})=\mathbf{0}.
\end{align}

The length of the code $V$ directly follows from the construction.

We show that covering radius $R_V$ of $V$ is equal to $R$.

Consider an arbitrary  column $\tT=(\f\s)\in\F_q^{\,r_0+Rm}$ with $\f\in\F_q^{\,r_0}$, $\s\in\F_q^{Rm}$,\\ $\s=(s_1,s_2,\ldots, s_{Rm})$, $s_i\in\F_q$. We partition $\s$ by $m$-vectors so that $\s=(S_0,S_1,\ldots,S_{R-1})$, $S_v=(s_{vm+1},s_{vm+2},\ldots,s_{vm+m})$, $v=0,1,\ldots,R-1$. We treat $S_v$ as an element of $\F_{q^m}$.

Since $V_{0}$ is an $[n_{0},n_{0}-r_{0}]_{q}R,R$ code, there exists a linear combination of the form
\begin{align}\label{eq4_proof_f}
 \f=\sum_{k=1}^R c_k\h_{j_k},~ c_k\in\F^*_q\text{ for all }k,
\end{align}
 see Definition \ref{def1_R,l code}. Now we can represent $\tT$ as a linear combination (with nonzero
coefficients) of $R$ distinct columns of $\Hh_V$. We have, see \eqref{eq4_Bj_QM_*},
\begin{align}\label{eq4_proof_t}
 \tT=\sum_{k=1}^R c_k\bb_{j_k}(x_k),~ c_k\in\F^*_q\text{ and }x_k\in\F_{q^m}\text{ for all }k,
\end{align}
where  values of $x_k$ are obtained from the linear system with nonzero determinant. If for $j_k$ in~\eqref{eq4_proof_f} we have $\beta_{j_k}\in
\F_{q^{m}}$ for all $k$, then the system has the form
\begin{align}\label{eq4_system_Fqm}
 \sum_{k=1}^R c_k\beta_{j_k}^{v}x_k=S_{v},~v=0,1,\ldots,R-1.
\end{align}
We put $0^0=1$.
If in \eqref{eq4_proof_f} we have, for example, $\beta_{j_R}=\ast$, then the system is as follows:
\begin{align}\label{eq4_system_*}
 \sum_{k=1}^{R-1} c_k\beta_{j_k}^{v}x_k=S_{v},~v=0,1,\ldots,R-2;\quad \sum_{k=1}^{R-1} c_k\beta_{j_k}^{R-1}x_k+c_Rx_R=S_{R-1}.
\end{align}

If $V_0$ is a locally optimal code, then every column $\h_j$ of $\Hh_0$ takes part in a representation of the form \eqref{eq4_proof_f}. If we remove  $\bb_{j_k}(\xi_u)$ from $\B_{j_k}$ then there is $(s_1,s_2,\ldots, s_{Rm})$ such that the system \eqref{eq4_system_Fqm} or \eqref{eq4_system_*} gives $x_k=\xi_u$; for some $\tT$ the representation~\eqref{eq4_proof_t} becomes impossible. So, all columns of $\Hh_V$ are essential and  $V$ is locally optimal.
 \end{proof}

\noindent \looseness=-1 \textbf{Construction QM$_\mathbf{2}$.} Let columns $\h_{j}$ belong to $\F_{q}^{\,r_{0}}$ and let
$\Hh_{0}=[\mathbf{h}_{1} \h_{2}\mathbf{\ldots
h}_{n_{0}}]$ be a
parity check matrix of an $[n_{0},n_{0}-r_{0}]_{q}R,\ell_0$
\emph{starting} code $V_{0}$ with $\ell_0=R-1$, $R\ge2$.  Let $m\geq 1$ be an integer
such that $q^m\ge n_0$. Let $\theta _{m,q}=\frac{q^{m+1}-1}{q-1}$. To each column
$\h_{j}$ we associate an element $\,\beta _{j}$ $\in
\F_{q^{m}}$ so that $\,\beta _{i}\neq \beta _{j}$
if $i\ne j$.  Let a new code $V$ be
the $ [n,n-(r_{0}+Rm)]_{q}R_{V},\ell_V$ code with $n=q^{m}n_{0}+\theta _{m,q}$
and parity check matrix $\Hh_{V}$ of the  form
\begin{align}
&\Hh_{V} =\left[\mathbf{C}~\B_{1}~\B_{2}~\ldots ~
\B_{n_{0}}\right] ,\quad \mathbf{ C=}\left[
\begin{array}{@{}c@{}}
\mathbf{0}_{r_{0}+(R-1)m} \\
\mathbf{W}_{m}
\end{array}
\right],  \label{eq4_QM-H_R-1}
\end{align}
where $\B_{j}$ is an $(r_{0}+Rm)\times q^m$ matrix as in \eqref{eq4_Bj_QM_*}, $\mathbf{C}$ is an $(r_{0}+Rm)\times \theta_{m,q}$ matrix,
$\mathbf{0}_{r_{0}+(R-1)m}$ is the zero $(r_{0}+(R-1)m)\times \theta_{m,q}$ matrix, $\mathbf{W}_{m}$ is a parity check $m\times\theta_{m,q}$ matrix of the $[\theta_{m,q},\theta_{m,q}-m,3]_{q}1$ Hamming code.

\begin{theorem}\label{th4_QM2}
In Construction QM$_2$, the new code $V$ with the parity check matrix \eqref{eq4_QM-H_R-1}, \eqref{eq4_Bj_QM_*} is an
$[n,n-(r_{0}+Rm),3]_{q}R,R$ surface-covering code with covering radius $R$ and length $n=q^{m}n_{0}+\frac{q^{m+1}-1}{q-1}$. Moreover,
if the starting code $V_0$ is locally optimal, then the new code $V$  is locally optimal too.
\end{theorem}
\begin{proof}
The length of the code $V$ directly follows from the construction.

The minimum distance is equal to 3 as the Hamming code is a code with $d=3$.

We show that covering radius $R_V$ of $V$ is equal to $R$.

Consider an arbitrary  column $\tT=(\f\s)\in\F_q^{\,r_0+Rm}$ with $\f\in\F_q^{\,r_0}$, $\s\in\F_q^{Rm}$,\\ $\s=(s_1,s_2,\ldots, s_{Rm})$, $s_i\in\F_q$. We partition $\s$ by $m$-vectors so that $\s=(S_0,S_1,\ldots,S_{R-1})$, $S_v=(s_{vm+1},s_{vm+2},\ldots,s_{vm+m})$, $v=0,1,\ldots,R-1$. We treat $S_v$ as an element of $\F_{q^m}$.

Since $V_{0}$ is an $[n_{0},n_{0}-r_{0}]_{q}R,\ell_0$ code with $\ell_0=R-1$, there exists a linear combination of $\varphi(\f)$ distinct columns of $\Hh_0$ of the form
\begin{align*}
 \f=\sum_{k=1}^{\varphi(\f)} c_k\h_{j_k},~ c_k\in\F^*_q\text{ for all }k, \varphi(\f)\in\{R-1,R\},
 \end{align*}
 see Definition \ref{def1_R,l code}. If $\varphi(\f)=R$ we act similarly to the proof of Theorem \ref{th4_QM1}.

 Let $\varphi(\f)=R-1$.  We represent $\tT$ as a linear combination (with nonzero
coefficients) of at most $R$ distinct columns of $\Hh_V$. We have, see \eqref{eq4_Bj_QM_*}, \eqref{eq4_QM-H_R-1},
\begin{align}\label{eq4_proof_t2}
 \tT=\eta\cc+\sum_{k=1}^{R-1} c_k\bb_{j_k}(x_k),~ c_k\in\F^*_q\text{ and }x_k\in\F_{q^m}\text{ for all }k,~\eta\in\F_q,
\end{align}
where $\cc$ is a column of $\mathbf{C}$ and $\eta=0$ means that the summand $\eta\cc$ is absent. Also, in~\eqref{eq4_proof_t2}, values of $x_k$ are obtained from the linear system
\begin{align*}
 \sum_{k=1}^{R-1} c_k\beta_{j_k}^{v}x_k=S_{v},~v=0,1,\ldots,R-2,
\end{align*}
with nonzero determinant. Finally, in \eqref{eq4_proof_t2}, $\cc=(\mathbf{0}\w$) where $\mathbf{0}$ is the zero $(r_0+(R-1)m)$-positional column and $\w$ is a column of $\W_m$ that satisfies the equality
\begin{align}\label{eq4_etaw}
\eta\w+\sum_{k=1}^{R-1} c_k\beta_{j_k}^{R-1}x_k=S_{R-1}.
\end{align}
In \eqref{eq4_etaw}, if $\sum_{k=1}^{R-1} c_k\beta_{j_k}^{R-1}x_k=S_{R-1}$ we have $\eta=0$. If $\sum_{k=1}^{R-1} c_k\beta_{j_k}^{R-1}x_k\ne S_{R-1}$, the needed column $\eta\w$ always exists as the Hamming code has covering radius 1.

Now we show that $V$  is an $[n,n-(r_{0}+Rm),3]_{q}R,R$ code, i.e. $\ell_V=R$. The critical case is when in \eqref{eq4_proof_t2} and \eqref{eq4_etaw} $\eta=0$, i.e. the summand $\eta\cc$ is absent. We use the approach of the proof of Theorem \ref{th4_QM1} regarding  \eqref{eq4_3word}. In \eqref{eq4_3word} we put $j=j_1,\xi_{u_1}=x_1,a=-c_1$ with $j_1,x_1,c_1$ taken from \eqref{eq4_proof_t2}. Then
\begin{align*}
&\tT=-c_1\bb_{j_1}(x_1)+b\bb_{j_1}(\xi_{u_2})+c\bb_{j_1}(\xi_{u_3})+\sum_{k=1}^{R-1} c_k\bb_{j_k}(x_k)=\sum_{k=2}^{R-1} c_k\bb_{j_k}(x_k)\\
&+b\bb_{j_1}(\xi_{u_2})+c\bb_{j_1}(\xi_{u_3}).
\end{align*}
Thus, we always can represent $\tT\in\F_q^{\,r_0+Rm}$ as a linear combination with nonzero coefficients of exactly $R$ columns of $\Hh_V$.

By above, if we remove any column of $\Hh_V$, some representation of $\tT$ becomes impossible. So, all columns of $\Hh_V$ are essential and the code $V$ is locally optimal.
 \end{proof}

\section{New infinite code families with fixed radius $R\ge4$ and increasing codimension $tR$}\label{sec5_inf_fam}

\looseness=-1 In the $\rho$-saturating set of Construction S \eqref{eq3_constr}--\eqref{eq3_points}, we consider a point $P_j$ (in homogeneous coordinates) as a column $\h_j$ of the parity check matrix $\widehat{\Hh}_\rho$ that defines the\\ $[qR+1,qR+1-2R,3]_{q}R,\ell$  code $\widehat{V}_\rho$ of covering radius $R=\rho+1$. To use Constructions QM$_{1}$ and QM$_{2}$ we show that $\ell=R-1$ if $q$ is even, and $\ell=R$ if $q$ is odd. This means that any column $\f$ of $\F_q^{\,2R}$ is equal to a linear combination with nonzero coefficients
of $R-1$ or $R$ columns of $\widehat{\Hh}_\rho$ for even $q$ and $R$ columns of $\widehat{\Hh}_\rho$ for odd $q$.

 We consider some properties of $\widehat{\Hh}_\rho$  useful to estimate $\ell$.
 Let $\f\in\F_q^{\,2R}$.  Let \\
 $J(\f)=\{\h_{j_1},\ldots,\h_{j_\beta}\}$  and $I_w=\{\h_{i_1},\ldots,\h_{i_w}\}$ be sets of distinct columns of $\widehat{\Hh}_\rho$ such that
  \begin{align}\label{eq5_Jf}
 &\f=\sum\limits_{k=1}^\beta c_k\h_{j_k},~\h_{j_k}\in J(\f)\text{ and } c_k\in\F^*_q\text{ for all }k;\displaybreak[3]\\
 &\label{eq5_Iw}
  \sum\limits_{k=1}^wm_k\h_{i_k}=\mathbf{0},~\h_{i_k}\in I_w \text{ and }m_k\in\F_q^*\text{ for all }k,~\mathbf{0}\in\F_q^{\,2R}\text{ is the zero column}.\displaybreak[3]\\
  &\text{By \eqref{eq5_Jf} and \eqref{eq5_Iw}, we have}\notag\displaybreak[3]\\
  &\label{eq5_Iw+Jf}
 \f=\sum\limits_{k=1}^\beta c_k\h_{j_k}+\mu\sum\limits_{k=1}^wm_k\h_{i_k},~\mu\in\F^*_q.
    \end{align}

Note that $I_w$ is a set of columns corresponding to a w\emph{eight $w$ codeword} of $\widehat{V}_\rho$.

 \looseness=-1   In the representation \eqref{eq5_Iw+Jf}, the number of distinct columns of $\widehat{\Hh}_\rho$, say $\beta^{\text{new}}$, depends on the intersection $I_w\cap J(\f)$ and the values of nonzero coefficients $c_k,m_k,\mu$, for example,
  \begin{align}\label{eq5_beta_new}
\beta^{\text{new}}=\left\{\begin{array}{lcl}
                 \beta+w & \text{if}& I_w\cap J(\f)=\emptyset;\\
                \beta+w-1 &  \text{if}&|I_w\cap J(\f)|=1,~ \h_{j_\beta}=\h_{i_w},~c_\beta+\mu m_w\ne0;\\
                 \beta+w-2 & \text{if}& |I_w\cap J(\f)|=1,~ \h_{j_\beta}=\h_{i_w},~c_\beta+\mu m_w=0;\\
                 \beta+w-2 & \text{if}& |I_w\cap J(\f)|=2,~ \h_{j_\beta}=\h_{i_w},~c_\beta+\mu m_w\ne0,\\
                 &&\hspace{2.69cm}\h_{j_{\beta-1}}=\h_{i_{w-1}},~c_{\beta-1}+\mu m_{w-1}\ne0.
               \end{array}
\right..
  \end{align}
To use \eqref{eq5_Iw+Jf}, \eqref{eq5_beta_new},
 submatrices of $\widehat{\Hh}_\rho$ can be treated as parity check matrices of codes; we call them \emph{component codes} and write in  Table~\ref{tab3_codes}, where $u=1,\ldots,\rho$, ``MDS'' notes a minimum distance separable code, ``AMDS'' says on an Almost MDS code.
\begin{table}[htb]
\caption{Components codes corresponding to submatrices of $\widehat{\Hh}_\rho$ based on \eqref{eq3_constr}--\eqref{eq3_points}}
  \begin{tabular}{@{}c@{\,}|@{~\,}l@{}|@{\,\,}c@{}|@{\,\,}l@{\,}|@{\,}c@{\,}|@{\,}c@{\,}|@{\,}c@{}}\hline
rows of $\widehat{\Hh}_\rho$&columns of $\widehat{\Hh}_\rho^{\vphantom{H^{H^H}}}$&
\renewcommand{\arraystretch}{1.0}$\begin{array}{c}
  \text{geometrical} \\
  \text{object}
\end{array}$
&code parameters &$q$&
\renewcommand{\arraystretch}{0.8}$\begin{array}{@{}c@{}}
  \text{code} \\
  \text{name}
\end{array}$&
\renewcommand{\arraystretch}{0.8}$\begin{array}{@{}c@{}}
  \text{code} \\
  \text{type}
\end{array}$
\\\hline
        1,2 &$\h_1\ldots\h_q$&$\{A_0^0\}\cup L_0^*$&$[q,q-2,3]_q2$& all&$\mathbb{L}_0$&MDS\\\hline
       $2u,2u+1,2u+2$&$\h_{qu+1}\ldots\h_{qu+q-1}$&$C_u^*$&$[q-1,q-4,4]_q3$& all&$\mathbb{C}_u$&MDS \\\hline
       $2u,2u+1,2u+2$&$\h_{qu+1}\ldots\h_{qu+q}$&$C_u^*\cup\{T_u\}$& $[q,q-3,4]_q3$& even&$\mathbb{C}_u^{T^{\vphantom{H}}}$&MDS \\\hline
       $2u,2u+1,2u+2$&$\h_{qu+1}\ldots\h_{qu+q}$&$C_u^*\cup\{T_u\}$& $[q,q-3,3]_q3$& odd&$\mathbb{C}_u^{T^{\vphantom{H}}}$&AMDS \\\hline
       $2\rho,2\rho+1,2\rho+2$&$\h_{q\rho+1}\ldots\h_{q\rho+q-1},$&$C_\rho^*\cup\{A_\rho^\infty\}$& $[q,q-3,4]_q3$&all&$\mathbb{C}_\rho^{\infty}$&MDS \\ 
       &$\h_{q\rho+q+1}$&&&& \\\hline
       $2\rho,2\rho+1,2\rho+2$&$\h_{q\rho+1}\ldots\h_{q\rho+q+1}$&$C_\rho^*\cup\{A_\rho^\infty,T_\rho\}$& $[q+1,q-2,4]_q3$&even&$\mathbb{C}_\rho^{\infty T^{\vphantom{H}}}$&MDS\\\hline
       $2\rho,2\rho+1,2\rho+2$&$\h_{q\rho+1}\ldots\h_{q\rho+q+1}$&$C_\rho^*\cup\{A_\rho^\infty,T_\rho\}$& $[q+1,q-2,3]_q3$&odd&$\mathbb{C}_\rho^{\infty T^{\vphantom{H}}}$&AMDS  \\\hline
  \end{tabular}
     \label{tab3_codes}
\end{table}
\begin{remark}\label{rem5a}
The following is useful to estimate $\ell$ in the   code $\widehat{V}_\rho$.

\textbf{(i)} In an $[n,n-r,d]_{q}$ MDS code, any $d$ columns of a parity check matrix correspond to a weight $d$ codeword \cite{MWS}.

\textbf{(ii)} In an $[n,n-r,d]_{q}$ MDS code with $n\le q$, there are codewords of \emph{all weights}\\ $w\in\{d,d+1,\ldots,n\}$ \cite[Th.\,6]{EzerGrasSoleMDS2011}.

\textbf{(iii)} If $q$ is odd, for AMDS component codes $\mathbb{C}_u^{T}$ and $\mathbb{C}_\rho^{\infty T}$, we note that $T_u$ lies on two tangents to $C_u$ (in  $A_u^0$ and $A_u^\infty$) and on $\frac{q-1}{2}$ bisecants of $C_u^*$. Every of these bisecants gives rise to a weight 3 codeword. The $(q-1)$-set of points of $C_u^*$ is partitioned to $\frac{q-1}{2}$ point pairs; every pair together with $T_u$ forms a weight 3 codeword.

\textbf{(iv)} From the  proofs of Sect. \ref{sec_constr} it can be seen that for the representation of a column $\f\in\F_q^{\,2R}$ it is sufficient to use (for every $u$) at most 3 points (columns) of $C_u^*$. Similarly, one can use  2 points of $\{A_0^0\}\cup L_0^*$. Therefore,  we have in $\{A_0^0\}\cup L_0^*$ and in every $C_u^*$ at least $q-4$ ``free" points (columns) that are not used to represent $\f$; these columns can be used to form sets $I_w$ useful to increase   $\beta^{\text{new}}$ for $\f$ by \eqref{eq5_Iw+Jf}, \eqref{eq5_beta_new}.

\looseness=-1\textbf{(v)} If $\beta<R$ in \eqref{eq5_Jf}, then at least $R-\beta$ component codes are not used to represent $\f$; the columns corresponding to these codes are ``free" and can be used to form sets $I_w$.

\textbf{(vi)}
If $q\ge7$, always  there exists  $\mu$ providing conditions ``$=0$", ``$\ne0$" in~\eqref{eq5_beta_new}.
\end{remark}
\begin{lemma}\label{lem5}
Let $q\ge7$. Let $R\ge4$.   Let\/ $\widehat{V}_\rho$ be the  $[Rq+1,Rq+1-2R,3]_{q}R,\ell$ locally optimal code such that the columns of its parity check matrix $\widehat{\Hh}_\rho$ correspond to  points (in homogeneous coordinates) of the minimal $\rho$-saturating set of Construction S \eqref{eq3_constr}--\eqref{eq3_points} with $\rho=R-1$. Then $\ell=R$ if $q$ is odd and $\ell=R-1$ if $q$ is even.
\end{lemma}
\begin{proof}
We should show that every column $\f$ of $\F_q^{\,2R}$ (including
the zero column) is equal to a linear combination with nonzero coefficients
of $R-1$ or $R$ columns of $\widehat{\Hh}_\rho$ for even $q$ and $R$ columns of $\widehat{\Hh}_\rho$ for odd $q$.

Let $I_w=\{\h_{i_1},\ldots,\h_{i_w}\}$ be a set of distinct columns of $\widehat{\Hh}_\rho$ corresponding to a weight $w$ codeword of an MDS component code.
Then there is a linear combination $L_w=\sum_{k=1}^wm_k\h_{i_k}=\mathbf{0}$, $m_k\in\F_q^*$, cf. \eqref{eq5_Iw}. Let $w_1+w_2+\ldots+w_b=T$. We denote $\Upsilon_T=L_{w_1}+L_{w_2}+\ldots+L_{w_b}=\mathbf{0}$ the sum of the linear combinations.

Let a column $\f\in\F_q^{\,2R}$ have the representation \eqref{eq5_Jf} of the form $\f=\sum_{k=1}^\beta c_k\h_{j_k}$ where $\h_{j_k}\in J(\f)$ and  $\beta\le R$.
If $\beta=R$, the assertions of the lemma hold.

 Let $0\le\beta\le R-3$ where $\beta=0$ corresponds to the zero column. We represent the column as $\f=\sum_{k=1}^\beta c_k\h_{j_k}+\Upsilon_{R-\beta}$ where the linear combinations $L_{w_j}$ of $\Upsilon_{R-\beta}$ consist of ``free" columns that are not used in the set $J(\f)$. We have several ``free" columns, see Remark \ref{rem5a}(iv),(v). The component code $\mathbb{L}_0$ has $d=3$. Therefore, taking into account also Remark \ref{rem5a}(i),(ii),
the sum $\Upsilon_{R-\beta}$ with $3\le R-\beta\le R$ always can be found.

Let $\beta\in\{R-2,R-1\}$. The increase of $\beta$ by $w-1$, $w-2$ is possible if some columns of $J(\f)$ and $I_w$ correspond to the same component code and $|I_w\cap J(\f)|\in\{1,2\}$, see \eqref{eq5_Iw+Jf}, \eqref{eq5_beta_new}. Let $d$ be minimum distance of a component code.  Due to Remark \ref{rem5a}(i),(iii),(iv), one always can take in \eqref{eq5_Iw} a  set $I_w$ with $w=d\in\{3,4\}$  so that $|I_w\cap J(\f)|\in\{1,2\}$. This provides the cases with $w=d=3$, $w-1=2$, $\beta^{\text{new}}=\beta+2$, and $w=d=4$, $w-2=2$, $\beta^{\text{new}}=\beta+2$.

 So, for even and odd $q$,
if $\beta=R-2$, we can obtain  $\beta^{\text{new}}=R$.

Let $\beta=R-1$. The case with $w=3$, $w-2=1$, $\beta^{\text{new}}=\beta+1$, can be provided if some column or a column pair of $J(\f)$ and $I_w$ correspond to the same code $\mathbb{L}_0$ (for all $q$) or to the same code $\mathbb{C}_u^{T}$, $\mathbb{C}_\rho^{\infty T}$ (for $q$ odd) since these codes have $d=3$. There exist columns $\f\in \F_q^{\,2R}$ such that  $\mathbb{L}_0$ is not used for their representation. Therefore we should consider only codes $\mathbb{C}_u^{T}$, $\mathbb{C}_\rho^{\infty T}$. For $q$ odd we always can obtain $\beta^{\text{new}}=R$ using $\mathbb{C}_u^{T}$, $\mathbb{C}_\rho^{\infty T}$ with $d=3$, see Remark \ref{rem5a}(iii).  But in general, for even $q$ (where MDS codes $\mathbb{C}_u^{T}$, $\mathbb{C}_\rho^{\infty T}$ have $d=4$) we are not able to do $\beta^{\text{new}}=R$ when $\beta=R-1$, see~\eqref{eq5_Iw+Jf}, \eqref{eq5_beta_new}.
\end{proof}

In Theorems \ref{th5_odd} and \ref{th5_even} we consider $R\ge4$ since for $R=2,3$, several short covering codes with $r=tR$ are given in detail in  \cite{DavParis,Dav95,DavCovRad2,DFMP-IEEE-LO,DGMP-AMC,DGMP_ACCT2008,DavOst-IEEE2001,DavOstEJC,DavOst-DESI2010}.

\begin{theorem}\label{th5_odd}
  Let $q\ge7$ be odd. Let $t$ be an integer. Then for all $R\ge4$ there is an infinite family of $[n,n-r,3]_qR,R$ locally optimal surface-covering codes with the parameters
  \begin{align*}
    n=Rq^{(r-R)/R}+q^{(r-2R)/R},~ r=tR,~t=2\text{ and }t\ge\lceil\log_q R\rceil+3.
  \end{align*}
\end{theorem}
\begin{proof}
  We take the $[Rq+1,Rq+1-2R,3]_{q}R,R$ code $\widehat{V}_\rho$, see Lemma \ref{lem5}, as the starting code $V_0$ of Construction QM$_1$. By Theorem \ref{th4_QM1}, we obtain an $[n,n-r,3]_q,R,R$ code with $n=(qR+1)q^m$, $r=2R+mR$. Obviously, $m+1=\frac{r-R}{R}$. The condition $q^m\ge n_0-1$ implies $q^m\ge qR$ whence $m\ge\lceil\log_q R\rceil+1$. Finally, we put $t=m+2$.
 \end{proof}
\begin{theorem}\label{th5_even}
  Let $q\ge8$ be even.  Let\/ $t$ be an integer. Let $m_1=\lceil\log_q (R+1)\rceil+1$. Then for all $R\ge4$ there are infinite families of $[n,n-r,3]_qR,R$ locally optimal surface-covering codes with the parameters
  \begin{align*}
    &\textbf{\emph{(i)}} ~n=Rq^{(r-R)/R}+2q^{(r-2R)/R}+\sum_{j=3}^{t}q^{(r-jR)/R},~r=tR,~m_1+2<t< 3m_1+2;
    \displaybreak[3]\\
    &\textbf{\emph{(ii)}} ~n=Rq^{(r-R)/R}+2q^{(r-2R)/R}+\sum_{j=3}^{m_1+2}q^{(r-jR)/R},\,r=tR,~t=m_1+2 \text{ and }t\ge3m_1+2.
  \end{align*}
\end{theorem}
\begin{proof}
  \textbf{(i)} We take the $[qR+1,qR+1-2R,3]_{q}R,\ell$ code $\widehat{V}_\rho$ with $\ell=R-1$, see Lemma \ref{lem5}, as the starting code $V_0$ of Construction QM$_2$. By Theorem \ref{th4_QM2}, we obtain an $[n,n-r,3]_q,R,R$ code with $n=(qR+1)q^m+\frac{q^{m+1}-1}{q-1}$, $r=2R+mR$.  Obviously, $m-(j-2)=\frac{r-jR}{R}$. The condition $q^m\ge n_0$ implies $q^m\ge qR+1$ whence $m\ge\lceil\log_q (qR+1)\rceil=\lceil\log_q (R+1)\rceil+1$. The restriction $m <3m_1$ is introduced as for $m\ge3m_1$ we have codes of (ii) that are better than ones in (i). For $m=m_1$, codes of (i) and (ii) are the same. Finally, we put $t=m+2$.

  \textbf{(ii)} In the relation (i), we put $t=m_1+2$ and obtain an  $[n_1,n_1-r_1,3]_qR,R$ code with $n_1=(qR+1)q^{m_1}+\frac{q^{m_1+1}-1}{q-1}$, $r_1=2R+m_1R$.  We take this code as the starting code $V_0$ of Construction~QM$_1$. By Theorem \ref{th4_QM1}, we obtain an $[n,n-r,3]_q,R,R$ code with $r=2R+m_1R+m_2R$, $q^{m_2}\ge n_1$,
  $  n=n_1q^{m_2}=(qR+1)q^{m_1+m_2}+\sum_{i=0}^{m_1}q^{m_1+m_2-i}.$
 Obviously, $m_1+m_2-i=\frac{r-(i+2)R}{R}$. Since $(R+1)q^{m_1+1}>n_1$, the condition $q^{m_2}\ge n_1$ is satisfied when $q^{m_2}\ge(R+1)q^{m_1+1}$ whence $m_2\ge\lceil\log_q (R+1)\rceil+m_1+1=2m_1$. Then we denote $2+m_1+m_2$ by $t$.
  \end{proof}
\section{New infinite code families with fixed even radius $R\ge2$ and increasing codimension $tR+\frac{R}{2}$}\label{sec_1sat_plane}
In the projective plane $\PG(2,q)$, a \emph{blocking} (resp. \emph{double blocking}) set $S$ is a set of points such that every line of $\PG(2,q)$ contains at least one (resp. two) points of $S$.

There is an useful connection
between double blocking sets and 1-saturating sets.
\begin{proposition}\label{prop6_2fold->1sat}
\emph{\cite[Cor.\,3.3]{DGMP-AMC}, \cite{KKKRS}} Let $q$ be a square. Any double blocking set in
the subplane $\PG(2,\sqrt{q}\,)\subset \PG(2,q)$ is a $1$-saturating set in the
plane $\PG(2,q)$.
\end{proposition}
\noindent In the following we shall use these results:
\begin{proposition}\label{prop6_2block}
\emph{\cite{BacHegSzon,BeuHegSzVoorde,BlokhLovStorSz,DGMP-AMC}}  Let $p$ be prime.  Let $\phi(q)$ be as in \eqref{eq2_phi}.
The following  bounds on the smallest size $\tau_2(2,q)$ of a double blocking set
 in $\PG(2,q)$ hold:
 \begin{align*}
 & \tau_2(2,q)\le2(q+q^{2/3}+q^{1/3}+1),&& q=p^{3},~p\le73    &&\emph{\cite[Th.\,3.5]{DGMP-AMC}};\displaybreak[3]\\
 & \tau_2(2,q)\le2(q+q^{2/3}+q^{1/3}+1),&& q=p^{3h},~p^h\equiv2\bmod7    &&\emph{\cite[Th.\,5.5]{BlokhLovStorSz}};\displaybreak[3]\\
 & \tau_2(2,q)\le2\left(q+\frac{q-1}{\phi(q)-1}\right),&& q=p^h,~h\ge2,~p\ge3&&\text{\emph{\cite[Cor.\,1.9]{BacHegSzon}}};\displaybreak[2]\\
 & \tau_2(2,q)\le2\left(q+\frac{q}{p}+1\right),&& q=p^h,~h\ge2,~p\ge7&&\text{\emph{\cite[Th.\,1.8, Cor.\,4.10]{BeuHegSzVoorde}}}.
 \end{align*}
\end{proposition}

Now we give a list of 1-saturating sets in the projective plane of square order. The sets (iv)--(vi) are new, they directly follow from Propositions \ref{prop6_2fold->1sat} and \ref{prop6_2block}.

\begin{proposition}\label{prop6_1sat}
  Let $q$ be a square. Let $p$ be prime. Let $\phi(\sqrt{q})$ be as in \eqref{eq2_phi}. Then in $\PG(2,q)$ there are  $1$-saturating sets of the following sizes:
\begin{align*}
&\textbf{\emph{(i)}}~~~ 3\sqrt{q}-1,&&q=p^{2h}\ge4,~h\ge1~~\,\emph{\cite[Th.\,5.2]{Dav95}};\displaybreak[3]\\
&\textbf{\emph{(ii)}}~~2\sqrt{q}+2\sqrt[4]{q}+2,&&q=p^{4h}\ge16,\,h\ge1~\emph{\cite[Th.\,3.3]{DGMP_ACCT2008}, \cite[Th.\,3.4]{DGMP-AMC}, \cite{KKKRS}};\displaybreak[3]\\
&\textbf{\emph{(iii)}}~2\sqrt{q}+2\sqrt[3]{q}+2\sqrt[6]{q}+2,&&q=p^{6},~p\le73\qquad\,\emph{\cite[Th.\,3.4]{DGMP_ACCT2008}, \cite[Cor.\,3.6]{DGMP-AMC}};\displaybreak[3]\\
&\textbf{\emph{(iv)}}~2\sqrt{q}+2\sqrt[3]{q}+2\sqrt[6]{q}+2,&&q=p^{6h},~p^h\equiv2\bmod7;\displaybreak[3]\\
&\textbf{\emph{(v)}}~~2\sqrt{q}+2\frac{\sqrt{q}-1}{\phi(\sqrt{q}\,)-1}\,,&&q=p^{2h},~h\ge2,~p\ge3;\displaybreak[3]\\
&\textbf{\emph{(vi)}}~2\sqrt{q}+2\frac{\sqrt{q}}{p}+2,&&q=p^{2h},~h\ge2,~p\ge7.
 \end{align*}
\end{proposition}
\begin{remark}\label{rem6_improve}
 In Proposition \ref{prop6_1sat}, if $\sqrt{q}=p^\eta$ with $\eta\ge3$ odd, then \emph{the new $1$-saturating sets} of (iv)--(vi) \emph{have smaller sizes than the known ones} of (i)--(iii).
 For example, if $q=p^6$, $\eta=3$, then the new size of (vi) is $2\sqrt{q}+2\sqrt[3]{q}+2$, cf.\,(iii).
 If $\eta\ge5$ odd, the known sets have size $3\sqrt{q}-1$ whereas new sizes are $2\sqrt{q}+o(\sqrt{q})$.  For example, if $q=p^{30}$, $\eta=15$, then the new size of
 (iv), (v) is $2\sqrt{q}+2\sqrt[3]{q}+2\sqrt[6]{q}+2$, cf.\,(i).
 In general, if $\eta\ge3$ is prime, then the case (vi)  gives smaller sizes than other variants.
If  $\eta$ is odd non-prime, then the variant~(v) is the best.

The case (iv) gives the same size as (v), if $3|\eta$. Therefore, in future we consider new codes and bounds resulting from Proposition \ref{prop6_1sat}(v),(vi).

Note also that if $q=p^2$, i.e. $\eta=1$, then the size (i) is the smallest in Proposition~\ref{prop6_1sat}. It is why we pay attention to this case, see Remarks \ref{rem6_Dav95}--\ref{rem6_2fold_prime} and Problem \ref{probl5} below.
\end{remark}
\begin{remark}\label{rem6_Dav95}
Let a point of $\PG(2,q)$ have the form $(x_0,x_1,x_2)$ where $x_i\in\F_q$, the leftmost nonzero  coordinate is equal to~1. Let $\beta$ be a primitive element of $\F_q$.

In \cite[Th.\,5.2,\,eq.\,(30)]{Dav95}, the following construction of a 1-saturating $(3\sqrt{q}-1)$-set $S$ in $\PG(2,q)$, $q$ square, is proposed:
\begin{align}\label{eq6_Dav95}
S=\{(1,0,x_2)|x_2\in\F_{\sqrt{q}}\}\cup\{(1,0,c\beta)|c\in\F_{\sqrt{q}}^*\}\cup\{(0,1,x_2)|x_2\in\F_{\sqrt{q}}\}.
\end{align}
We describe this construction in more detail than in \cite{Dav95} using, for the description, the Baer sublines similarly to \cite[Prop.\,3.2]{BorSzonTicDefinSets}. In \cite{Dav95}, see \eqref{eq6_Dav95}, specific Baer sublines are noted. Here we explain the structure and role of these specific sublines.
Two Baer subplanes $\mathcal{B}_1$ and $\mathcal{B}_2$ are considered. In the points of $\mathcal{B}_1$, all coordinates $x_i\in\F_{\sqrt{q}}$. Also, $\mathcal{B}_2=\mathcal{B}_1\Phi$ where $\Phi$ is the collineation such that $(x_0,x_1,x_2)\Phi=(x_0,x_1\beta,x_2\beta)$.
Let $L_i\subset\PG(2,q)$ be the ``long'' line of equation $x_i=0$. Let $L_{i,j}=L_i\cap\mathcal{B}_j$ be the Baer subline of $L_i$ in the Baer subplane $\mathcal{B}_j$. We denote points $A_1=(0,0,1)$, $A_2=(1,0,0)$. Obviously, $\{A_1,A_2\}\subset\mathcal{B}_1\cap\mathcal{B}_2$.

We have $L_{0,1}=L_{0,2}$, $\mathcal{B}_1\cap\mathcal{B}_2=L_{0,1}\cup\{A_2\}$. Thus, the  Baer subplanes $\mathcal{B}_1$ and $\mathcal{B}_2$ have the common Baer subline $L_{0,1}$ and also the common point $A_2$ not on $L_{0,1}$. Also, $L_{0,1}\cap L_{1,1}\cap L_{1,2}=\{A_1\}$. So, we consider three Baer sublines through $A_1$; one of them $L_{0,1}$ is common for $\mathcal{B}_1$ and $\mathcal{B}_2$; the other two ($L_{1,1}$ and $L_{1,2}$) belong to the same long line $L_{1}$ that passes through $A_2\notin L_{0,1}$ and $A_1\in L_{0,1}$. The needed set consists of these three Baer sublines without their intersection point, i.e. $S=(L_{0,1}\cup L_{1,1}\cup L_{1,2})\setminus \{A_1\}$. Since  $L_{1,1}\cap L_{1,2}=\{A_1,A_2\}$ it holds that $|S|=3\sqrt{q}-1$. Note that if $A_1$ is not removed from $S$ then we have no bisecants of $S$ through $A_1$.

\looseness=-1 All points on $L_0$ and $L_1$ are 1-covered by $S$. Consider a point $A=(1,a,b)\notin(L_0\cup L_1)$ with $a=a_1\beta+a_0\in\F_q^*,$ $b=b_1\beta+b_0\in\F_q$. (If $a=0$ then $A\in L_1$.) Let $a_0\ne0$. Then $A=(1,0,(b_1-a_1a_0^{-1}b_0)\beta)+a(0,1,a_0^{-1}b_0)$. Let $a_0=0$. Then $a_1\ne0$ and $A=(1,0,b_0)+a(0,1,a_1^{-1}b_1)$. Thus, $A$ is 1-covered by $S$. Also, from the above consideration it follows that all points of $S$ are 1-essential and $S$ is a \emph{minimal} 1-saturating set.
\end{remark}
\begin{remark}
 In \cite[Ex.\,B]{Ughi} and \cite[Prop.\,3.2]{BorSzonTicDefinSets}, constructions of a 1-saturating $3\sqrt{q}$-set in $\PG(2,q)$, $q$ square, are proposed. In \cite{Ughi},  the set is minimal; it consists of three non-concurrent Baer sublines in a Baer subplane. In \cite{BorSzonTicDefinSets}, the set  is non-minimal; it is similar to one of the construction \cite[Th.\,5.2]{Dav95}, see its description in Remark~\ref{rem6_Dav95}. However, in \cite{BorSzonTicDefinSets}, the intersection point of the three Baer sublines is not removed from the 1-saturating set.
\end{remark}
\begin{remark}\label{rem6_2fold_prime}
Let $p$ be prime. To construct a 1-saturating $(3p-1)$-set in $\PG(2,p^2)$ one can apply Proposition \ref{prop6_2fold->1sat} to a double blocking set in $\PG(2,p)$. However, double blocking $(3p-1)$-sets in $\PG(2,p)$ are known only for $q=13,
  19,31,37,43$, see \cite{CsajHeger}. Moreover, in $\PG(2,p)$, no double blocking sets of size less than $3p-1$ are known.
\end{remark}

In $\PG(2,p^2)$, $p$ prime, by \cite[Tab.\,2]{DGMP-AMC}, we have the following sporadic examples of $1$-saturating $k$-sets with $k<3p-1$: $p^2=9,k=6$; $p^2=25,k=12$; $p^2=49,k=18$.

\begin{problem}\label{probl5}
 Develop a general construction of a $1$-saturating $k$-set in $\PG(2,p^2)$, $p$ prime, such that $k<3p-1$.
\end{problem}
In \cite{DavCovRad2,DGMP-AMC}, a lift-construction is given. It provides the following result.
\begin{proposition}\label{prop6_induct}
 \emph{\cite[Ex.\,6]{DavCovRad2}, \cite[Th.\,4.4]{DGMP-AMC}}
  Let an $ [n_{q},n_{q}-3]_{q}2$ code exist. Let $n_q<q$ and $q+1\geq 2n_q$. Let  $f_q(r,2)$ be as in \eqref{eq2_fqrR}. Then there is an infinite family of
$[n,n-r]_{q}2$ codes with odd codimension $r=2t+1\geq3$, $t\ge1$, and length
$  n=n_{q}q^{(r-3)/2}+2q^{(r-5)/2}+f_q(r,2).$
\end{proposition}

\begin{theorem}\label{th6_codesR=2}
Assume that $p$ is prime, $q=p^{2h},~h\ge2$, and covering radius $R=2$. Let $\phi(\sqrt{q}\,)$ and $f_q(r,2)$ be as in \eqref{eq2_phi}, \eqref{eq2_fqrR}. Then there exist infinite families of $[n,n-r]_{q}2$ codes with odd codimension $r=2t+1\geq3$, $t\ge1$, and length
\begin{align*}
 &n=\left(2+2\frac{\sqrt{q}-1}{\sqrt{q}(\phi(\sqrt{q}\,)-1)}\right)q^{(r-2)/2}+2\lfloor q^{(r-5)/2}\rfloor+f_q(r,2),~p\ge3;\displaybreak[3]\\
 &n=\left(2+\frac{2}{p}+\frac{2}{\sqrt{q}}\right)q^{(r-2)/2}+2\lfloor q^{(r-5)/2}\rfloor+f_q(r,2),~p\ge7.
\end{align*}
\end{theorem}
\begin{proof}
  Let $n_q$ be the size of the 1-saturating sets of Proposition \ref{prop6_1sat}(iii),(iv). We treat every point (in homogeneous coordinates) of the set as a column of an $3\times n_q$ parity check matrix of an $[n_q,n_q-3]_q2$ code. For these codes it can be shown that $n_q<q$ and  $q+1\geq 2n_q$. Then we use Proposition~\ref{prop6_induct}.
 \end{proof}
The direct sum construction \cite[Sect.\,4.2]{DGMP-AMC} gives the following lemma.
\begin{lemma}\label{lem6_evenR}
Let covering radius $R\ge2$ be even. Let an $ [n'',n''-r'']_{q}2$ code exist. Then there is an
$[\frac{R}{2}n'',\frac{R}{2}n''-\frac{R}{2}r'']_{q}R$ code.
\end{lemma}

   \begin{theorem}\label{th6_evenR_r=tR+R/2}
Assume that $p$ is prime, $q=p^{2h},~h\ge2$, $R\ge2$ even, and code codimension is $r=tR+\frac{R}{2}$ with integer $t\ge1$. Let $\phi(\sqrt{q}\,)$ and $f_q(r,R)$ be as in \eqref{eq2_phi},~\eqref{eq2_fqrR}. Then for all even $R\ge2$ there are infinite families of $[n,n-r]_qR$ codes with fixed covering radius $R$, codimension $r=tR+\frac{R}{2}$, $t\ge1$, and length
\begin{align*}
&n= R\left(1+\frac{\sqrt{q}-1}{\sqrt{q}(\phi(\sqrt{q}\,)-1)}\right)q^{(r-R)/R}+R\left\lfloor  q^{(r-2R)/R-0.5}\right\rfloor+\frac{R}{2}f_q(r,R),~p\ge3;\displaybreak[3]\\
&n=R\left(1+\frac{1}{p}+\frac{1}{\sqrt{q}}\right)q^{(r-R)/R}+R\left\lfloor q^{(r-2R)/R-0.5}\right\rfloor+\frac{R}{2}f_q(r,R),~p\ge7.
\end{align*}
  \end{theorem}
\begin{proof}
  We take codes of Theorem \ref{th6_codesR=2} as the codes $[n'',n''-r'']_{q}2$ of Lemma \ref{lem6_evenR}.
 \end{proof}

\end{document}